\RequirePackage{fix-cm}
\documentclass[smallextended,numbook]{svjour3-pp}       % onecolumn (second format)
\smartqed  % flush right qed marks, e.g. at end of proof
\usepackage{graphicx, natbib}
\usepackage{amsmath, amsfonts,amssymb}%, amsthm}
\usepackage{dsfont}
\usepackage[normalem]{ulem}
\numberwithin{equation}{section}
\newcommand{\ind}{{\mathds{1}}}
\begin{document}

\title{Realistic extensions of a Brownian ratchet for protein
  translocation\thanks{Supported by the BMBF through FRISYS
    (Kennzeichen 0313921) } }


\author{A. Depperschmidt \and N. Ketterer \and \\ 
  P. Pfaffelhuber }


\institute{A. Depperschmidt \at
University of Freiburg\\
\email{depperschmidt@stochastik.uni-freiburg.de}           %  \\
\and N. Ketterer \at 
University of Freiburg\\
\email{nicolasketterer@gmx.de}
\and P. Pfaffelhuber\at 
University of Freiburg\\
\email{p.p@stochastik.uni-freiburg.de}
}

\date{Received: date / Accepted: date}

\maketitle

\begin{abstract}
  We study a model for the translocation of proteins across membranes
  through a nanopore using a ratcheting mechanism. When the protein
  enters the nanopore it diffuses in and out of the pore according to
  a Brownian motion. Moreover, it is bound by ratcheting molecules
  which hinder the diffusion of the protein out of the nanopore, i.e.\
  the Brownian motion is reflected such that no ratcheting molecule
  exits the pore.  New ratcheting molecules bind at rate
  $\gamma$. Extending our previous approach
  \citep{depper_pfaffel:2010} we allow the ratcheting molecules to
  dissociate (at rate $\delta$) from the protein (Model~I).  We also
  provide an approximate model (Model~II) which assumes a Poisson
  equilibrium of ratcheting molecules on one side of the current
  reflection boundary. Using analytical methods and simulations we
  show that the speed of both models are approximately the same. Our
  analytical results on Model~II give the speed of translocation by
  means of a solution of an ordinary differential equation.

\end{abstract}

\section{Introduction}
\label{intro}
Most proteins are generated within the cellular cytosol and have to be
transported where they are needed. Frequently, they have to be
transported across membranes, e.g.\ into the endoplasmatic reticulum,
mitochondria or other organelles, or they are secreted, i.e.\
transported across the cell wall \citep{NeupertBrunner2002}. Protein
translocation occurs either \emph{co-translationally}, e.g.\ when
ribosomes attach to the membrane of the endoplasmatic reticulum and
proteins are transported into the lumen through a nanopore already
during translation, or \emph{post-translationally}
\citep{Rapoport2007}.

Consider translocation into \emph{mitochondria} first. Over 99\% of
mitochondrial proteins are translocated post-translationally
\citep{Wickner2005} and the molecular mechanisms have been studied in
detail. When a substrate is translocated into the mitochondrium, it
has to cross the translocase outer membrane (TOM) and the translocase
inner membrane (TIM). The mitochondrial heat shock protein 70
(mtHsp70) is known to play a crucial role in translocation of
substrates \citep{Glick1995}. mtHsp70 changes between its ATP-bound
form, where it has an open pocket for binding to the substrate, and the
ADP-bound form where the pocket is closed
\citep{NeupertBrunner2002,NeupertHermann2007}. There are two differing
opinions about the role of mtHsp70 in translocation. Either mtHsp70
pulls \emph{actively} the substrate through the nanopore, or it only
prevents backsliding of the substrate, leading to a \emph{passive}
mechanism
\citep{Glick1995,NeupertBrunner2002,NeupertHermann2007}. Such a
ratcheting mechanism was first introduced by \cite{Schneider1994} and
\cite{Okamoto2002}. The best argument for active pulling is that the
protein which is translocated must be unfolded outside the
mitochondria, so translocation occurs against a force
\citep{Pfanner2002}. However, experimental evidence is still lacking
and binding of mtHsp70 to the substrate can explain observations
\citep{NeupertHermann2007}.

Many proteins are translocated into the \emph{endoplasmatic reticulum}
already during translation. Moreover, post-translational translocation
has been suggested to occur through a passive ratcheting mechanism in
eukaryotic cells by \cite{Matlack1999}. Here, the nanopore is formed
by the Sec63 complex and the ratcheting molecules are called BiP which
is a member of the Hsp70 family. BiP exists in an ATP-bound state with
an open binding pocket. This is activated by an interaction with the
J-domain of Sec63, which can also close the binding pocket after BiP
is bound to the substrate. Interestingly, BiP binds unspecifically to
the substrate and therefore is able to mediate translocation of many
different proteins \citep{Rapoport2007}. For protein translocation
into \emph{chloroplasts} in plants, ratcheting mechanisms have not yet
been considered yet \citep{Strittmatter2010,Li2010}.

~

Quantitative descriptions of cellular ratcheting mechanisms began with
pioneering work of \cite{Simon_al:1992} and \cite{Peskin_al:1993}. In their
model, ratcheting molecules can bind to the protein at the nanopore and hinder
diffusion out of the nanopore completely (perfect ratchet) or only with a
probability $p<1$ (imperfect ratchet). The assumption that ratcheting molecules can
only bind directly at the pore is motivated by translocation into the
endoplasmatic reticulum where the ratcheting molecule, BiP, is known to
interact with Sec61 which builds the pore. However, the interaction with the
nanopore can only mean that concentration of ratcheting molecules is higher
near the pore. (See \cite{Griesemer2007}, for a mathematical model of the
interaction of the ratcheting molecule with the nanopore.)  We take the same
approach as \cite{Orsogna2007} and assume that ratcheting molecules may bind
uniformly along the protein. This assumes that the concentration of ratcheting
molecules is approximately constant between the pore and the first bound
molecules.

Most quantitative descriptions of the Brownian ratchet assume a finite
length of ratcheting molecules (e.g.\
\citealp{Zandi2003,Orsogna2007}). \cite{Ambjornsson:2005} describe a
\emph{car parking effect} arising by the distribution of ratcheting
molecules along the substrate. They assume a fast binding and
unbinding of ratcheting molecules, leading to an effective probability
any position of the protein is bound. The case of fast binding is
studied in \cite{Fricks2005}, who derive a law of large numbers and a
central limit theorem in the case that Brownian motion has a
drift. Effects of binding affinities which depend on the protein
sequence have been put forward in
\cite{Abdolvahab2008,Abdolvahab2011}.

~

The goal of this paper is to introduce a realistic model for the ratcheting
mechanism for protein translocation in mitochondria and the endoplasmatic
reticulum as described above. \cite{Elston:2002} tried to decide if
translocation happens actively (by pulling the protein) or passively (by a
ratcheting mechanism) using likelihood ratio tests. However, after fitting the
model parameters using least squares, their statistical tests were not
significant and both models fitted experimental data well. New molecular
techniques have to give rise to better data to decide between active an
passive transport. In the present paper, we obtain predictions for the speed
of translocation which will shed light on empirical studies in the future. 

Our model of the Brownian ratchet is similar to the ratchet studied in
\cite{Liebermeister:2001}, but is continuous, assumes free diffusion
of the protein and binding of ratcheting molecules anywhere along the
protein. Since ratcheting molecules can dissociate from the protein,
we call the process a \emph{broken Brownian ratchet}. We give two
models, termed Model~I and Model~II. For Model~I, which takes into
account every ratcheting molecule at all times, we can only show that
the speed of translocation is positive
(Theorem~\ref{T:speedI}). Model~II serves as an approximation. When a
ratcheting molecule dissociates, we assume that binding and unbinding
of ratcheting molecules is in equilibrium. This model is easier to
analyze but has about the same behavior for the speed of translocation
(see the simulation results in Figure~\ref{fig3}). For Model~II, we
can show a law of large numbers giving the speed of transport in terms
of a solution of a differential equation, Theorem~\ref{T:speedII}.

\section{The model}
The model we study extends the Brownian ratchet introduced in
\cite{depper_pfaffel:2010}.  Namely, protein translocation across a
membrane is based on the following assumptions:
\begin{setlength}{\leftmargini}{1cm} 
  \renewcommand{\labelenumi}{(\roman{enumi})}
  \begin{enumerate}
  \item The protein moves in and out with the same probability.
  \item The protein movement is reflected at binding ratcheting
    molecules.
  \item The protein is infinitely long.
  \item The ratcheting molecules are infinitely small.
  \item Ratcheting molecules may bind to the protein at a continuum of
    sites.
  \item The ratcheting molecules may dissociate from the protein.
\end{enumerate}
\end{setlength}
These assumptions are exactly the same as in
\cite{depper_pfaffel:2010} except for the last. They assumed that the
dissociation rate of the ratcheting molecules from the protein is much
smaller than their binding rate to the protein, leading to effectively
no dissociation of ratcheting molecules.

First, we translate the above set of assumptions into a mathematical
model. By $X_t$, we denote the total length of the protein which
already crossed the membrane between times $0$ and $t$. In addition,
$R_t\leq X_t$ is the largest distance of $X_t$ to a ratcheting
molecule. We assume that the protein has already moved across the
membrane by time $0$, such that ratcheting molecules may bind between
$X_t$ and $-\infty$ at constant rate. Consider the pair $(X_t,
R_t)_{t\geq 0}$, where $(X_t)_{t\geq 0}$ is a Brownian motion, started
in $X_0=x$, reflected at $R_t := \sup \mathcal R_t$ and $\mathcal R_t
\subseteq (-\infty, X_t]$ is the random set of all bound ratcheting
molecules at time $t$. The dynamics of $\mathcal R_t$ is as follows:
starting in $\mathcal R_0 = \{0\}$, a point at $dx$ is added at rate
$\gamma \ind_{x\in (-\infty,X_t]} dx$ (i.e.\ a transition from
$\mathcal R_{t-}$ to $\mathcal R_t := \mathcal R_{t-}\cup\{x\}$
occurs). In addition, every $x\in\mathcal R_t$ is deleted at rate
$\delta$ (i.e.\ a transition from $\mathcal R_{t-}$ to $\mathcal R_t
:= \mathcal R_{t-}\setminus \{x\}$ occurs at rate $\delta \ind_{x\in
  \mathcal R_{t-}}$). Note that the last mechanism with rate $\delta$
models the dissociation of ratcheting molecules from the protein (see (vi)
above). We rely on the following graphical construction, which is
illustrated in Figure~\ref{fig1}.

\begin{definition}[\boldmath$\gamma/\delta$-broken Brownian ratchet, Model
  I]\label{modelI}
  \sloppy Let $\mathcal{N}^{\gamma,\delta}$ be a Poisson point process
  on $[0,\infty)\times \mathbb R \times [0,\infty)$ with intensity
  measure $\gamma \delta e^{-\delta z} \lambda^3(d\tau ,dr, dz)$,
  conditioned on $\mathcal{N}^{\gamma, \delta}(\{0\}\times
  \{0\}\times[0,\infty))=1$, where $\gamma >0, \delta\geq 0$ and
  $\lambda^3$ denotes the Lebesgue measure on $[0,\infty)\times\mathbb
  R \times [0,\infty)$. Moreover, $\mathcal{N}^\gamma$ is the
  projection of $\mathcal{N}^{\gamma,\delta}$ on the first two
  coordinates.  Let $(B_t^0)_{t\geq 0}, (B_t^1)_{t\geq 0},\dots $ be a
  sequence of independent Brownian motions, independent of
  $\mathcal{N}^{\gamma,\delta}$, with $B_0^0 = B_0^1 = \cdots =0$. We
  define recursively times $t_0, t_1,\dots$ when the reflection
  boundary changes as well as triples $(\tau_0, r_0, z_0), (\tau_1,
  r_1, z_1), (\tau_2, r_2, z_2),\dots\in \mathcal{N}^{\gamma,\delta}$
  such that $r_n$ is the reflection boundary in the interval $[t_n,
  t_{n+1})$. Define $t_0 = \tau_0 = r_0 =0$, $z_0:=z$ if $(0,0,z) \in
  \mathcal{N}^{\gamma,\delta}$ and for $x_0 \ge 0$ set $X_t := |x_0+B_t^0|$
  for $t_0\leq t < t_1$, and the set of new reflection boundaries above
  $r_n$ before time $\tau_n+z_n$, 
  \begin{equation}
    \begin{aligned}
      \Gamma_n & := \{(\tau,r,z) \in \mathcal{N}^{\gamma,\delta}: t_n \leq  \tau< \tau_n
      + z_n, r_n\leq r < X_\tau\}, \\ 
      \inf \Gamma_n & := (\tau,r,z) \text{ if } (\tau,r,z)\in \Gamma_n
      \text{ and }\tau = \inf \pi_1 \Gamma_n, 
    \end{aligned}
  \end{equation}
  where $\pi_1$ is the projection on the first coordinate. Now, we set
  recursively for $n=0,1,2,\dots$
  \begin{equation}
    \label{eq:taun1}
    \begin{aligned}
      t_{n+1} & := (\tau_n + z_n)\wedge \inf \pi_1\Gamma_n,\\
      (\tau_{n+1}, & \, r_{n+1}, z_{n+1}) := \begin{cases} \inf
        \Gamma_n, &
        \text{ if } \Gamma_n\not=\emptyset, \\[1ex]
        (\tau,r,z) & \text{ if }\Gamma_n=\emptyset \text{ and } \\ &
        \quad r = \max\{r\leq X_\tau: \tau \leq \tau_n+z_n <
        \tau+z\}.
      \end{cases}
    \end{aligned}
  \end{equation}
  Finally,
  \begin{equation}
    \left.
      \begin{aligned}
        R_t & := r_n,\\
        X_t & := R_t + |B_{t-t_n}^n + X_{t_n-} - R_t|
      \end{aligned}
    \right\} \text{ for $t_n \leq t < t_{n+1}$.}
  \end{equation}
  We refer to the process $(X_t, R_t)_{t\geq 0}$ as the
  \emph{$\gamma/\delta$-broken Brownian ratchet, Model~I} with initial
  value $(X_0, R_0)=(x_0,0)$.
  % \noindent
  % The $\gamma/\delta$-broken Brownian ratchet, Model~I, can as well be
  % started in $x> 0$ by using a Brownian motion $(B_t^0)_{t\geq 0}$
  % starting in $x$.
\end{definition}

\begin{remark}[Interpretation]
  In the definition above, $(\tau,r,z) \in
  \mathcal{N}^{\gamma,\delta}$ represents a ratcheting molecule which
  binds at time $\tau$ at position $r$ to the protein and stays bound
  up to time $\tau+z$. We start at time $t=0$ with one bound molecule
  at position 0. The variables $t_1, t_2,\dots$ represent times when
  the reflection boundary changes. Note that $R_t$, at any time $t$,
  is the position of the ratcheting molecule which is bound (i.e.\
  $\tau\leq t < \tau+z$) and closest to $X_t$.  By time $t$, let us
  denote by $(\tau,r,z)\in \mathcal{N}^{\gamma,\delta}$ the active
  point if $R_t=r$ (and recall that all points in
  $\mathcal{N}^{\gamma,\delta}$ have all coordinates different, almost
  surely). By definition, $t_n$ is the time when the active point
  changes for the $n$th time. There are two possibilities when the
  active point $(\tau,r,z)$ changes. First, a new Poisson point can
  fall between $r$ and $X_t$ (as described by the first line
  in~\eqref{eq:taun1}). Second, when $t=\tau+z$, the current active
  point dissociates, and the next active point is the one closest to
  $X_t$ (given by the second line in~\eqref{eq:taun1}). Last, note that
  $X_{t_n} = X_{t_n-}$ by construction, for all $n$, meaning that
  $(X_t)_{t\geq 0}$ is continuous and note that $|B_{t-t_n}^k +
  X_{t_n-} - R_t|$ is a Brownian motion, started at $X_{t_n-} - R_t$,
  reflected at 0, such that $(X_t)_{t\geq 0}$ is reflected at $R_t$ at
  all times.
  % See
  % Figure~\ref{fig1} for an illustration of the graphical
  % construction
  % of the $\gamma/\delta$-Brownian ratchet, Model~I.

  The graphical construction of $(X_t)_{t\geq 0}$ is done step by step
  between jump times of the reflection boundary, $(R_t)_{t\geq 0}$,
  i.e.\ when the active point changes. This works since, after any
  finite time $t$ there may only be a finite number of jumps of the
  reflection boundary so that the construction works for any large
  time.
\end{remark}

\begin{figure}[htb]
  \begin{center}
    \includegraphics[width=9cm]{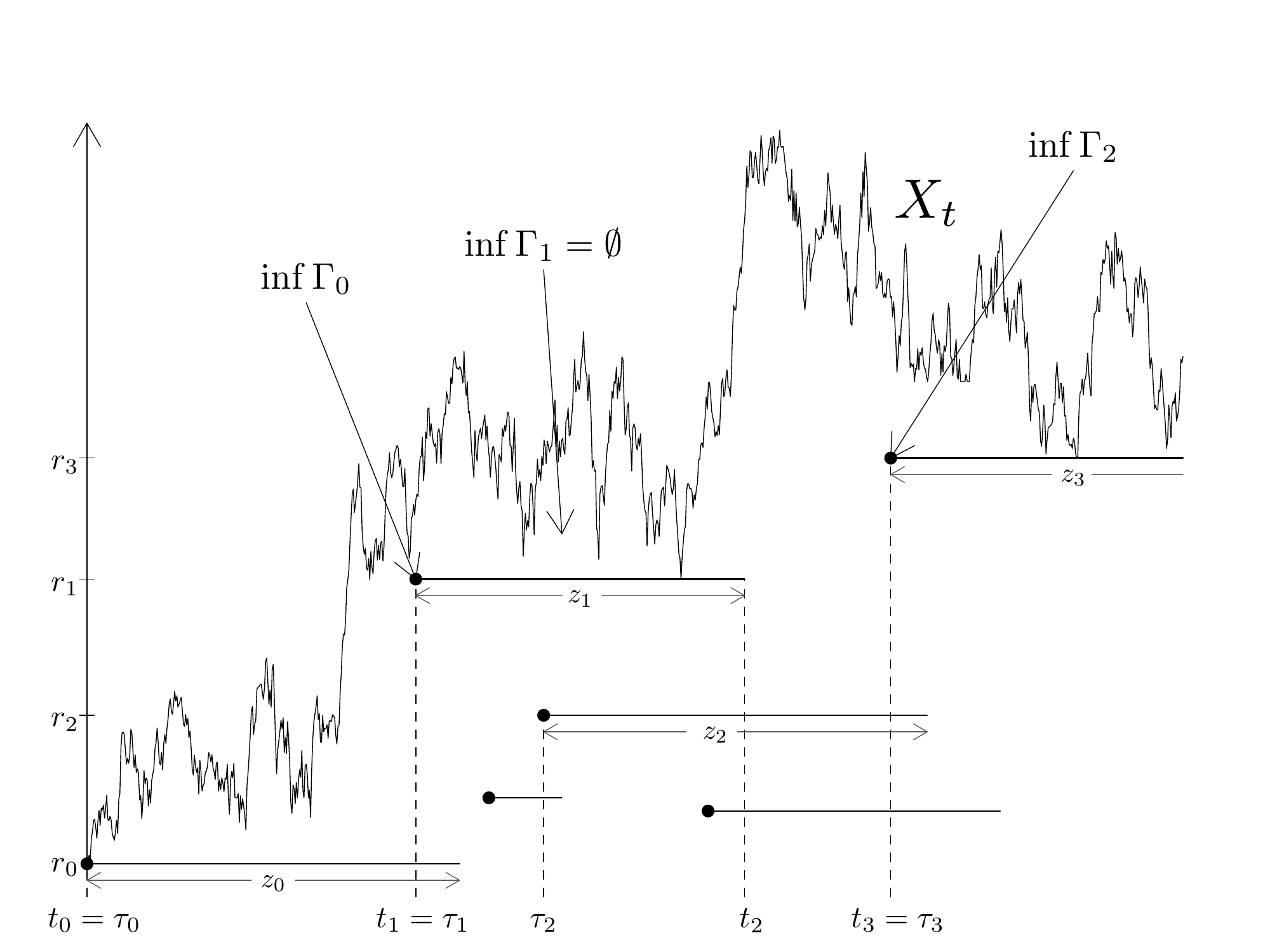}%
    \caption{\label{fig1} One realization of the graphical
      construction of Model~I. The path of $(X_t)_{t\geq 0}$ is
      reflected at the boundary $(R_t)_{t\geq 0}$.}
  \end{center}
\end{figure}

\noindent
For Model~I we show the following result.

\begin{theorem}[Speed of the broken Brownian ratchet, Model~I]
  Let \label{T:speedI} $(X_t,R_t)_{t\geq 0}$ be a broken Brownian
  ratchet, Model~I. Then, there are constants $0<c\leq C < \infty$,
  where $c$ depends on $\gamma$ and $\delta$ and $C$ only depends on
  $\gamma$ such that
  \[ c \leq \underline a_{\gamma,\delta} :=\liminf_{t\to\infty}
  \frac{X_t}{t} \leq \overline a_{\gamma,\delta}
  :=\limsup_{t\to\infty} \frac{X_t}{t} \leq C\] almost
  surely. Moreover, $\underline a_{\gamma,\delta} = \gamma^{1/3}
  \underline a_{1, \delta/\gamma^{2/3}}$ and $\overline
  a_{\gamma,\delta} = \gamma^{1/3} \overline a_{1,
    \delta/\gamma^{2/3}}$.
\end{theorem}

\begin{remark}[Interpretation]
  The result says that the speed of the broken Brownian ratchet,
  Model~I, (if it exists as $a_{\gamma,\delta} := \underline
  a_{\gamma,\delta} = \overline a_{\gamma,\delta}$) is positive, no
  matter how large $\delta$ is. In addition, it scales with
  $\gamma^{1/3}$ like $a_{\gamma,\delta} = \gamma^{1/3} a_{1,
    \delta/\gamma^{2/3}}$.
\end{remark}

~

\noindent
Although the speed of the broken Brownian ratchet is positive by
Theorem~\ref{T:speedI}, we aim at a more complete picture. However,
the difficulty in the analysis of Model~I is that it is not local in
the sense that a single Poisson point can be active more than once. In
particular, dependencies between Poisson points arise as time
evolves. For this reason, we consider a second model with similar
properties as Model~I. Briefly speaking, we introduce a Model~II by
assuming that the possible reflection boundaries below the currently
active one are always in their equilibrium. 

~

\begin{figure}[htb]
  \begin{center}
    \includegraphics[width=9cm]{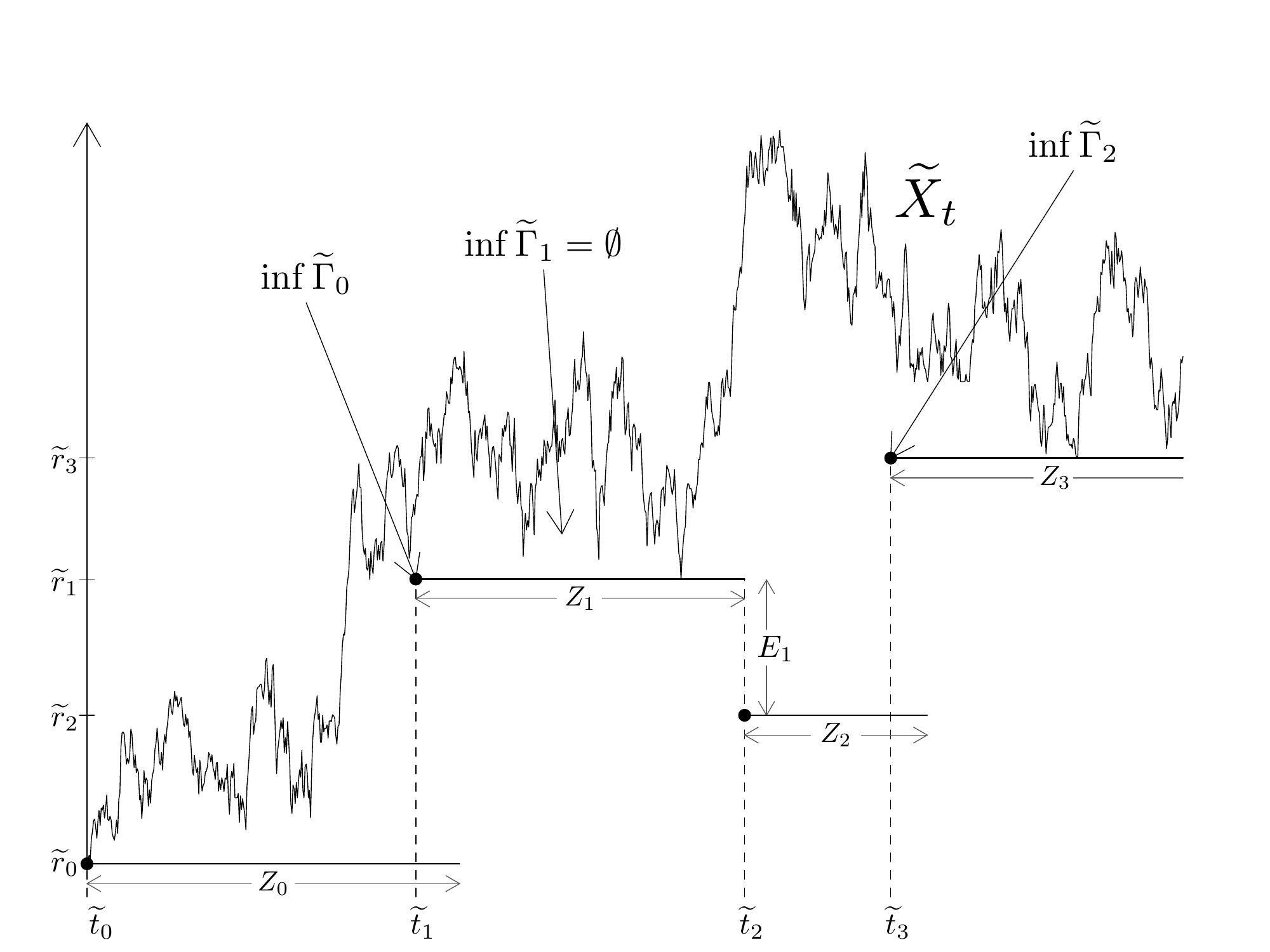}% 
    \caption{\label{fig2} One realization of the graphical
      construction of Model~II. The path of $(\widetilde X_t)_{t\geq
        0}$ is reflected at the boundary $(\widetilde R_t)_{t\geq
        0}$. Random variables $Z_0, Z_1,...$ are independent and
      Exp$(\delta)$-distributed, while $E_0, E_1,...$ are independent
      and Exp$(\gamma/\delta)$-distributed. }
  \end{center}
\end{figure}

\begin{remark}[Motivation]
  Consider Model~I and assume $(\tau,r,z)$ is the active point by time
  $t$. Consider a time $t'>t$ when the active point changes and $t' =
  \tau+z$. As an approximation, we assume that the set \[\{r':
  (\tau',r',z')\in \mathcal{N}^{\gamma,\delta} \text{ for some
  }\tau',z'\in[0,\infty), \tau'\leq t'\leq \tau'+z'\}\] is in its
  stationary distribution which arises for $z\to\infty$, i.e.\ if the
  last active point stays for a long time. This stationary
  distribution is readily computed. It is clear that only points
  $r'<r$ are possible. Moreover, the probability that a point in $dr'$
  is in the set equals 
  \[ \int_0^\infty \gamma e^{-\delta z}dz dr' = \frac \gamma \delta dr'. \]
  So, the set is distributed according to a Poisson process with
  intensity measure $\tfrac \gamma\delta \ind_{(-\infty, r]}(r')
  \lambda(dr')$. As a consequence, if the active point vanishes, the
  reflection boundary jumps down an exponentially distributed amount
  with rate $\gamma/\delta$ in the approximate model. This leads to
  the following Model~II, which is illustrated in Figure~\ref{fig2}.
\end{remark}

\begin{definition}[\boldmath$\gamma/\delta$-Broken Brownian ratchet,
  Model~II]\label{modelII}
  Let $\widetilde {\mathcal{N}}^{\gamma}$ be a Poisson point process on
  $[0,\infty)\times \mathbb R$ with intensity measure $\gamma
  \lambda^2(dt,dx)$, conditioned on
  $\widetilde{\mathcal{N}}^{\gamma}(\{0\}\times \{0\})=1$, where $\gamma >0$
  and $\lambda^2$ denotes the Lebesgue measure on $[0,\infty)\times\mathbb
  R$. Moreover, $Z_0, Z_1,\dots$ are independent exponentially distributed
  with parameter $\delta\geq 0$ and $E_0, E_1,\dots$ are independent and
  exponentially distributed with parameter $\gamma/\delta$ (they are not
  needed if $\delta=0$). Let $(\widetilde B_t^0)_{t\geq 0}, (\widetilde 
  B_t^1)_{t\geq 0},\dots $ be a sequence of independent Brownian motions with
  $\widetilde B_0^0 = \widetilde B_0^1 = \dots =0$. Define $\widetilde t_0 =
  \widetilde r_0 =0$ and for $x_0 \ge 0$ set $\widetilde X_t :=
  |x_0+\widetilde B_t^0|$ for $\widetilde t_0\leq t < \widetilde t_1$, and
  \begin{equation}
    \begin{aligned}
      \widetilde\Gamma_n & := \{(\tau,r) \in \widetilde{\mathcal{N}}^{\gamma}:
      \widetilde t_n \leq  \tau< \widetilde t_n + Z_n, r_n\leq r < X_\tau\}, \\
      \inf \widetilde\Gamma_n & := (\tau,r) \text{ if } (\tau,r)\in
      \Gamma_n \text{ and }\tau = \inf \pi_1 \Gamma_n,
    \end{aligned}
  \end{equation}
  where $\pi_1$ is the projection on the first coordinate. Now, we set
  recursively for $n=0,1,2,...$
  \begin{equation}
    \label{eq:taun}
    \begin{aligned}
      \widetilde t_{n+1} & := (\widetilde t_n+Z_n) \wedge
      \inf\pi_1\widetilde \Gamma_n,\\
      \widetilde r_{n+1} & = \begin{cases} r & \text{ if } \widetilde
        \Gamma_n \neq\emptyset \text{ and }(\widetilde t_{n+1},r) =
        \inf\widetilde\Gamma_n, \\\widetilde r_n - E_n, & \text{ if 
        }\widetilde \Gamma_n = \emptyset. %\\
      \end{cases}
    \end{aligned}
  \end{equation}
  Finally,
  \begin{equation}
    \left.
      \begin{aligned}
        \widetilde R_t & := \widetilde r_n,\\
        \widetilde X_t & := \widetilde R_t + |\widetilde B_{t-\widetilde t_n}^k
        + \widetilde X_{\widetilde t_n-} - \widetilde R_t|
      \end{aligned}
    \right\} \text{ for $\widetilde t_n \leq t < \widetilde t_{n+1}$.} 
  \end{equation}
  We refer to the process $(\widetilde X_t, \widetilde R_t)_{t\geq 0}$
  as the \emph{$\gamma/\delta$-broken Brownian ratchet, Model~II} with initial
  value $(\widetilde X_0,\widetilde R_0)=(x_0,0)$.
\end{definition}

\begin{theorem}[Speed of the broken Brownian ratchet, Model~II]
  Let \label{T:speedII} $(\widetilde X_t,\widetilde R_t)_{t\geq 0}$ be
  a $\gamma/\delta$-broken Brownian ratchet, Model~II. Then,
  \[ \lim_{t\to\infty} \frac{\widetilde X_t}{t} =
  -\frac{A'(0)}{2A(0)}\] almost surely, where $A$ is the first
  coordinate of a solution of the system
  \begin{align}
    \label{eq:diff-system}
    \begin{split}
      A''(z) &= - 2\delta B(z) + 2\gamma z A(z), \\
      B'(z) & = -A'(z) - \frac\gamma{\delta} B(z)
    \end{split}
  \end{align}
  such that $A(0)=1/2$, $B(0) =0$ and $A$ is strictly decreasing with
  $A(z) \to 0$ as $z \to \infty$.
\end{theorem}

\begin{remark}[Simulations, uniqueness of \boldmath$A$, and the case
  $\delta=0$]\label{rem:26}
  \begin{enumerate}
  \item Since Model~II is only a convenient approximation of Model~I,
    we use simulations to see differences in the speed of Model~I
    and~II. By scaling properties of both models (see
    Propositions~\ref{P:scaleI} and~\ref{P:scaleII}), we require
    simulations only for a single parameter $\gamma$.  As can be seen
    in Figure~\ref{fig3}, the speed of both models is almost the
    same. The fact that Model~II is faster for low $\delta$ can be
    explained: we assume that the number of possible reflection
    boundaries below the currently active one is in equilibrium, which
    means that these are more than for Model~I, where the equilibrium
    is not yet attained. Since more reflection boundaries mean that
    the broken Brownian ratchet moves faster, the speed of Model~II is
    higher. For high values of $\delta$, Model~I shows a higher speed
    in our simulations. The reason is that we fixed the reflection
    boundary at 0 in our simulations of Model~I, and hence the
    simulations overstimate the speed of protein translocation.
  \item For a proper use of Theorem~\ref{T:speedII} in
    Figure~\ref{fig3}, we need to solve the
    system~\eqref{eq:diff-system}. We searched numerically for a
    solution by trying out various values for $A'(0)$ and determined
    which value approximately leads to $A(z)\to 0$ as $z\to
    \infty$. In our numerical analysis, we found only a single value
    of $A'(0)$ with this property, for all $\delta$. Hence, we
    strongly conjecture there is a unique solution of the
    system~\eqref{eq:diff-system} satisfying $A(0)=1/2$, $B(0)=0$ and
    $A(z)\to 0$ as $z\to\infty$.
  \item In the case $\delta=0$, the first equation
    in~\eqref{eq:diff-system} reads $A''(z)-2\gamma zA(z)=0$. The only
    solution with the required boundary values is given by $z\mapsto
    Ai((2\gamma)^{1/3} z)/(2Ai(0))$, where $Ai$ is the Airy function
    (i.e.\ solution of $u''(z) - zu(z)=0$) going to 0 as
    $z\to\infty$. By well-known properties of $Ai$ (see e.g.\
    \citealp{AbramowitzStegun:1972}), the speed is thus given by
    \begin{equation}
      \label{eq:Cgamma}
      \begin{aligned}
        \lim_{t\to\infty} \frac{X_t}{t} & = - \frac{Ai'(0)}{2Ai(0)}
        (2\gamma)^{1/3} % = \frac{3^{2/3}\Gamma(2/3)}{2\cdot
        %3^{1/3}\Gamma(1/3)}(2\gamma)^{1/3} \\ & 
        = \frac{\Gamma(2/3)}{\Gamma(1/3)}
        \Big(\frac{3\gamma}{4}\Big)^{1/3} \approx 0.36 \cdot
        (2\gamma)^{1/3},
      \end{aligned}
    \end{equation}
    a result known from Theorem~1 in \cite{depper_pfaffel:2010}.
  \end{enumerate}
\end{remark}

\begin{figure}[htb]
  \begin{center}
    \includegraphics[width=8.5cm]{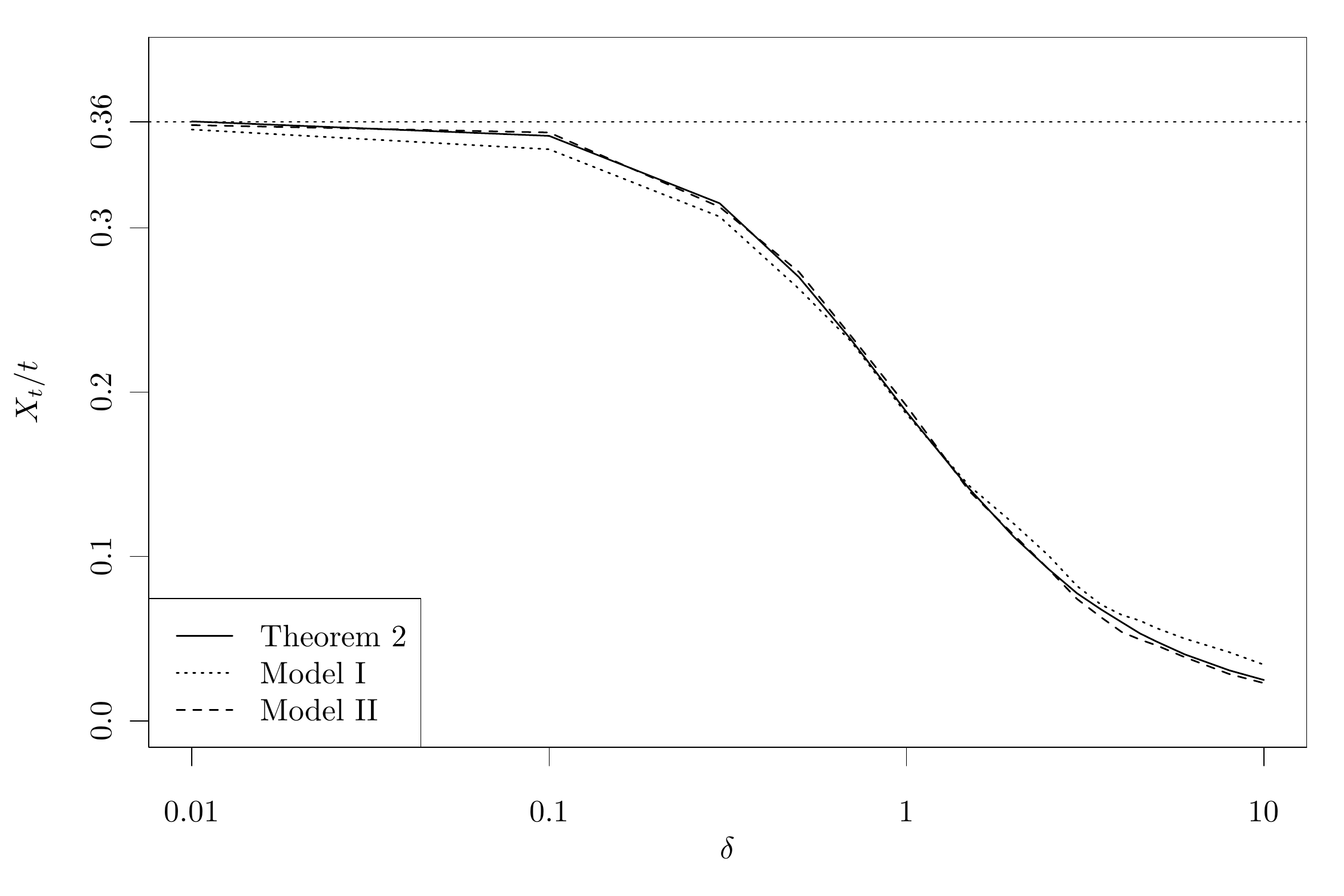}%
    \caption{\label{fig3} A simulation study for the speed of the broken
      Brownian ratchet, Model~I and Model~II. The parameter $\gamma=\tfrac 12$
      is fixed. See also Remark~\ref{rem:26}. The horizontal line at 0.36 represents the 
      speed in the case $\delta=0$ from~\eqref{eq:Cgamma}.
      % , but can be transformed to any other
      %value by Propositions~\ref{P:scaleI} and~\ref{P:scaleII}.
    }
  \end{center}
\end{figure}

\section{Proofs: Model~I}
In this section, we prove Theorem~\ref{T:speedI}. Three ingredients
are needed: First, a scaling property for the $\gamma/\delta$-broken
Brownian ratchet, Model~I, is obtained in
Section~\ref{sub:scaleI}. Second, an upper bound for the speed can
trivially be obtained by setting $\delta=0$ and using results from
\cite{depper_pfaffel:2010}. Third, a lower bound is obtained by
thinning Model~I, which we do in Section~\ref{sub:thinning}. We
conclude the proof in Section~\ref{sub:proofI}.

\subsection{Scaling property}
\label{sub:scaleI}
Here we obtain a scaling property of the $\gamma/\delta$-broken Brownian
ratchet, Model~I, which is based on a rescaling of space and time in
Definition~\ref{modelI}.

\begin{proposition}[Scaling property, Model~I]\label{P:scaleI}
  Let $(X_t^{\gamma;\delta}, {R}_t^{\gamma;\delta})_{t\geq 0}$ be the
  $\gamma/\delta$-broken Brownian ratchet, Model~I, with $\gamma>0$,
  $\delta\geq 0$ and initial value $(0,0)$. Then 
  \begin{align}\label{skalierung}
    (X_t^{\gamma;\delta}, R_t^{\gamma;\delta})_{t\geq 0}\overset{d}{=}
    \gamma^{-\frac{1}{3}}\left(X_{\gamma^{2/3}t}^{1;\delta/\gamma^{2/3}},
      R_{\gamma^{2/3}t}^{1;\delta/\gamma^{2/3}}\right)_{t\geq 0},
  \end{align}
  where \, $\stackrel d =$ \,  denotes equality in distribution.
\end{proposition}

\begin{proof}
  \sloppy Using the notation from Definition~\ref{modelI}, we need to
  understand what (\ref{skalierung}) means for the Poisson process
  $\mathcal{N}^{\gamma,\delta}$ and for the Brownian motions
  $(B_t^0)_{t\geq 0}, (B_t^1)_{t\geq 0},\dots$  We use the linear
  rescaling of time and space
  \begin{align*}
    g_\gamma: \begin{cases} {\mathbb R}^3 & \to {\mathbb R}^3 \\
      (t,x,z) & \mapsto
      (\gamma^{2/3}t,\gamma^{1/3}x,\gamma^{2/3}z).\end{cases}
  \end{align*}
  \sloppy and obtain $g_\gamma(\mathcal N^{\gamma,\delta}) \stackrel d
  = \mathcal N^{1,\delta/\gamma^{2/3}}$. Moreover, we have that
  $(\sqrt cB^k_t)_{t\geq 0} \overset d = \big(B^k_{ct}\big)_{t\geq 0}$
  for all $c>0$ and $k=0,1,2,\dots$ and so, if $\widetilde g_\gamma:
  \mathbb R^2 \to\mathbb R^2$ consists of the first two coordinates of
  $g_\gamma$,
  \begin{equation}
    \label{eq:BB1}
    \begin{aligned}
      \widetilde g_\gamma\left((t,{B}_t^k)_{t\geq 0}\right)&=
      \left(\big(\gamma^{2/3}t,\gamma^{1/3}{B}_t^k\big)_{t\geq 0}\right)\\
      &\overset{d}{=}\left(\big(\gamma^{2/3}t,{B}^k_{\gamma^{2/3}t}\big)_{t\geq
          0}\right)
      =\left(\big(s, B_s^k\big)_{s\geq 0}\right)
    \end{aligned}
  \end{equation}
  for $k=0,1,2,\dots$ where $s:=\gamma^{2/3}t$. Hence, the distribution
  of the Brownian motions are unaffected by the rescaling. We have
  shown that
  \[\gamma^{1/3}(X_t^{\gamma,\delta}, R_t^{\gamma,\delta})_{t\geq 0}\stackrel
  d =  (X_{\gamma^{2/3} t}^{1,\delta/\gamma^{2/3}},R_{\gamma^{2/3} 
    t}^{1,\delta/\gamma^{2/3}})_{t\geq 0}\] 
  and the assertion follows. \qed 
\end{proof}

\subsection{Lower bound}
\label{sub:thinning}
\sloppy Consider the possible reflection boundaries $[\tau, \tau+z)
\times \{r\}$ for all $(\tau,r,z) \in \mathcal{N}^{\gamma,\delta}$ in
Definition~\ref{modelI}. Clearly, the speed of the
$\gamma/\delta$-ratchet decreases if we take less possible reflection
boundaries, leading to a lower bound for the speed of the broken
ratchet. We will use the following thinned version of Model~I. Note
that the formal difference to Model~I is the restriction
$\sup\Sigma_{n+1}\leq\tau$ in \eqref{eq:widehattaun}. An illustration
can be found in Figure~\ref{fig4}.

\begin{definition}[\boldmath$\gamma/\delta$-broken Brownian ratchet, thinned
  Model~I]
  Consider the same probability space as in Definition~\ref{modelI},
  with the same $\mathcal{N}^{\gamma, \delta}, \mathcal{N}^\gamma,
  (B_t^0)_{t\geq 0}, (B_t^1)_{t\geq 0}, (B_t^2)_{t\geq 0},\dots$
  Again, we define recursively times $\widehat t_0, \widehat
  t_1,\dots$ when the reflection boundary changes as well as triples
  $(\widehat \tau_0, \widehat r_0, \widehat z_0), (\widehat \tau_1,
  \widehat r_1, \widehat z_1), (\widehat \tau_2, \widehat r_2,
  \widehat z_2),\dots\in \mathcal{N}^{\gamma,\delta}$ such that
  $\widehat r_n$ is the reflection boundary in the interval $[\widehat
  t_n, \widehat t_{n+1})$. Define $\Sigma_{0} := \emptyset$, $\widehat
  t_0 = \widehat \tau_0 = \widehat r_0 =0$, $\widehat z_0:=z$ if
  $(0,0,z) \in \mathcal{N}^{\gamma,\delta}$ and for $x_0 \geq 0$ set 
  $\widehat X_t := |x_0 + B_t^0|$ for $\widehat t_0\leq t < 
  \widehat t_1$, and
  \begin{equation}
    \begin{aligned}
      \widehat \Gamma_n & := \{(\tau,r,z) \in \mathcal{N}^{\gamma,\delta}:
      \widehat t_n \leq
      \tau< \widehat \tau_n + \widehat z_n, \widehat r_n\leq r < \widehat
      X_\tau\} \\ 
      \inf \widehat \Gamma_n & := (\tau,r,z) \text{ if } (\tau,r,z)\in
      \widehat \Gamma_n \text{ and }\tau = \inf \pi_1 \widehat \Gamma_n
    \end{aligned}
  \end{equation}
  where $\pi_1$ is the projection on the first coordinate. Now, we set
  recursively for $n=0,1,2,\dots$
  \begin{equation}
    \label{eq:widehattaun}
    \begin{aligned}
      \widehat t_{n+1} & := (\widehat\tau_n + \widehat z_n)\wedge
      \inf\pi_1\widehat\Gamma_n,\\
      \Sigma_{n+1} & := \Sigma_{n} \cup \{\inf \{t\in [\widehat t_{n},
      \widehat t_{n+1}):
      \widehat X_t=\widehat R_t\}\} \\
      (\widehat \tau_{n+1}, & \, \widehat r_{n+1}, \widehat z_{n+1})
      \\ & :=
      \begin{cases} \inf \widehat \Gamma_n, &
        \!\!\text{if }\widehat \Gamma_n\not=\emptyset \\[1ex]
        (\tau,r,z) & \!\!\text{if }\widehat \Gamma_n=\emptyset \text{
          and } \\ & \quad r = \max\{r\leq \widehat X_\tau: \sup
        \Sigma_{n+1} \leq \tau \leq \widehat \tau_n + \widehat z_n <
        \tau+z\}.
      \end{cases}\\
    \end{aligned}
  \end{equation}
  Finally,
  \begin{equation}
    \left.
      \begin{aligned}
        \widehat R_t & := \widehat r_n,\\
        \widehat X_t & := \widehat R_t + |B_{t - \widehat t_n}^k +
        \widehat X_{\widehat t_n-} - \widehat R_t|
      \end{aligned}
    \right\} \text{ for $\widehat t_n \leq t < \widehat t_{n+1}$.} 
  \end{equation}
  We refer to the process $(\widehat X_t, \widehat R_t)_{t\geq 0}$ as
  the \emph{$\gamma/\delta$-broken Brownian ratchet, thinned Model~I} with initial
  value $(\widehat X_0, \widehat R_0)=(x_0,0)$. 
\end{definition}

\begin{figure}[htb]
  \begin{center}
    \includegraphics[width=9cm]{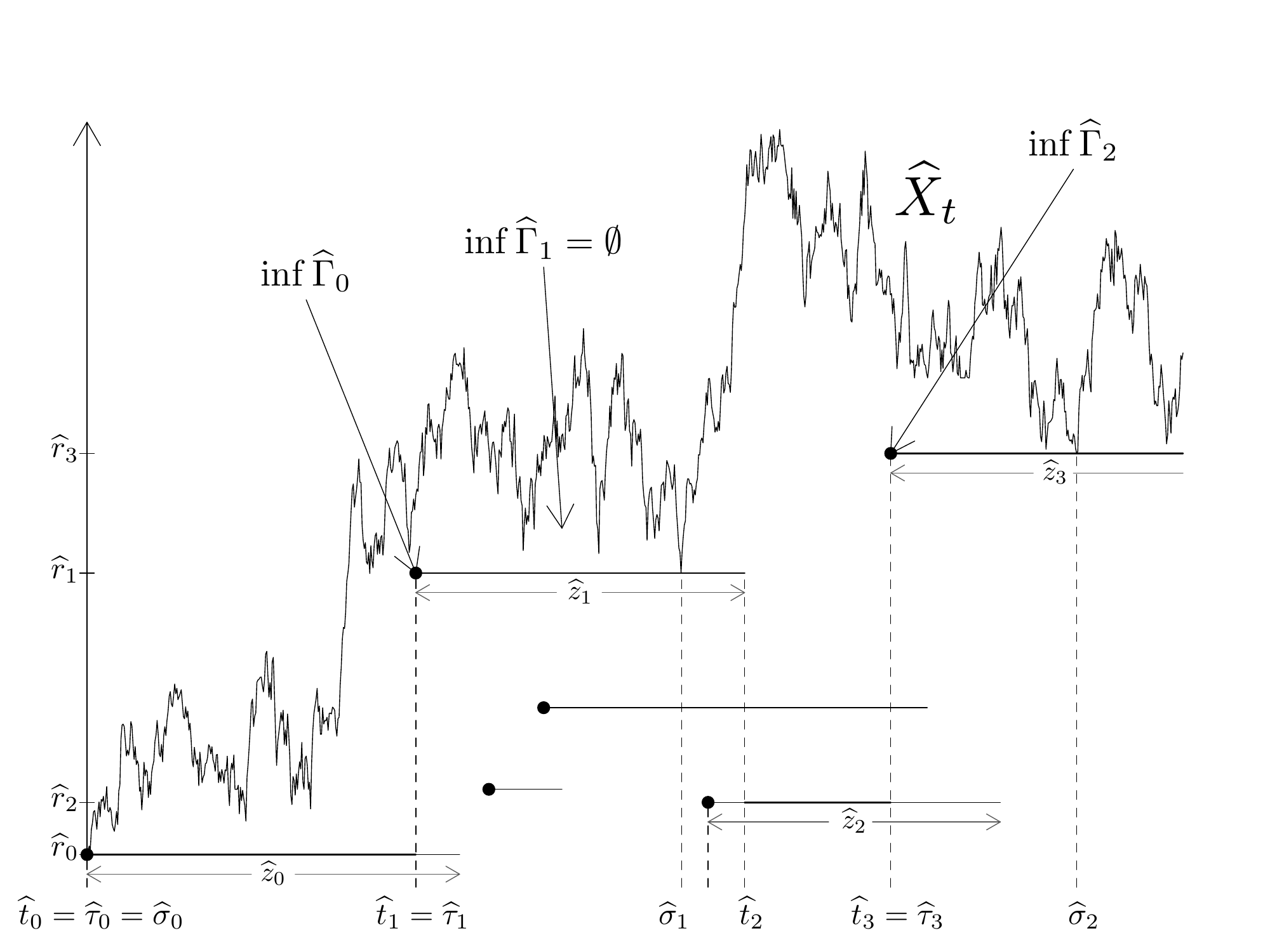}%
    \caption{\label{fig4} One realization of the graphical
      construction of thinned Model~I. In contrast to Model~I, there
      is the set of times $\Sigma_\infty = \{\widehat \sigma_0,
      \widehat \sigma_1, \widehat\sigma_2,...\}$ which serve as
      renewal points for the process. For example, note that the
      reflection boundary between $\widehat t_2$ and $\widehat t_3$ is
      lower than in Figure~\ref{fig1}.}
  \end{center}
\end{figure}

\begin{remark}[Renewal times for the thinned Model~I]
  Again, $\widehat t_0, \widehat t_1, \widehat t_2, \dots$ are times
  when the reflection boundary changes for the thinned
  Model~I. Moreover, the set $\Sigma_n$ consists of all first times
  after changing the reflection boundary, when $X_t=R_t$, up to
  $\widehat t_{n}$. Clearly, the additional restriction $\sup
  \Sigma_{n+1} \leq \tau$ in the definition of $(\widehat t_{n+1},
  \widehat \tau_{n+1}, \widehat z_{n+1})$ in~\eqref{eq:widehattaun}
  leads to taking less Poisson points and hence $\widehat R_t \leq
  R_t$ and $\widehat X_t\leq X_t$ for all $t\geq 0$, almost
  surely. Moreover, the times $\Sigma_\infty := \bigcup_{n=0}^\infty
  \Sigma_n$ are renewal times for the process $(\widehat X_t-\widehat
  R_t)_{t\geq 0}$: We have $\widehat X_t - \widehat R_t = 0$ for
  $t\in\Sigma_\infty$ and after time $t$, any reflection boundary $r$
  with $(\tau,r,z)\in \mathcal{N}^{\gamma,\delta}$ must fulfill
  $\tau\geq t$ in the thinned Model~I. In other words, the possible
  Poisson points for the reflection boundary before and after time $t$
  are distinct, so the distributions of $(\widehat X_s - \widehat
  R_s)_{0\leq s < t}$ and $(\widehat X_s - \widehat R_s)_{s\geq t}$ 
  are independent.
\end{remark}

\noindent
Using the renewal structure of the thinned Model~I allows us to define
a cumulative process.

\begin{lemma}[Renewal structure]
  Let $(\widehat X_t, \widehat R_t)_{t\geq 0}$ be a broken Brownian
  ratchet, thinned Model~I, and $\Sigma_\infty = \{\widehat \sigma_0,
  \widehat\sigma_1,\dots\}$. Then, $(\widehat\sigma_n -
  \widehat\sigma_{n-1},\widehat X_{\widehat\sigma_n} - \widehat
  X_{\widehat\sigma_{n-1}})_{n=1,2,\dots}$ is a sequence of bivariate
  iid random variables.
\end{lemma}

\begin{proof}
  By the construction of the thinned Model~I, possible reflection
  boundaries $r$ with $(\tau,r,z) \in \mathcal{N}^{\gamma,\delta}$ are
  only used in one interval $[\widehat\sigma_0, \widehat\sigma_1),
  [\widehat\sigma_1, \widehat\sigma_2), [\widehat\sigma_2,
  \widehat\sigma_3), \dots$ Hence, the Poisson points used for
  constructing $(\widehat\sigma_n - \widehat\sigma_{n-1},\widehat
  X_{\widehat\sigma_n} - \widehat X_{\widehat\sigma_{n-1}})_{n\geq 1}$
  arise independently, which  shows the result. \qed 
\end{proof}

\begin{remark}[The thinned Model~I as a cumulative
  process]\label{remark_LLN} \leavevmode \\ 
  Let $(\widehat X_t, \widehat R_t)_{t\geq 0}$ be the
  $\gamma/\delta$-Brownian ratchet, thinned Model~I. For
  $\Sigma_\infty = \{\widehat\sigma_0, \widehat\sigma_1,\dots\}$, set
  \begin{align}\label{cumul_process}
    M_t:=\min\{n:\widehat\sigma_n>t\},\quad
    S_n:=\sum_{i=1}^n{(\widehat X_{\widehat\sigma_i} - \widehat
      X_{\widehat\sigma_{i-1}})}, 
  \end{align}
  so that we have $\widehat X_t=S_{M_t}+A_t$, where
  \begin{align}
    \label{eq:cumul_proc_At}
    A_t:=\widehat X_{\widehat\sigma_0} + \widehat X_t - \widehat 
    X_{\widehat\sigma_{M_t}}. 
  \end{align}
  According to the definition of cumulative processes in \cite{Roginsky:1994},
  the thinned broken Brownian ratchet $\mathcal{X}$ is a type A cumulative
  process with remainder $A_t$. Note that finite first moments of
  $\widehat\sigma_1 - \widehat\sigma_0$ and
  $X_{\widehat\sigma_1}-X_{\widehat\sigma_0}$ are sufficient for the strong
  law of large numbers for $S_{M_t}$. We will see that a law of large numbers
  holds also for $(\widehat X_t)_{t\geq 0}$, as we will show in
  Proposition~\ref{remainder} that $A_t/t$ converges to 0 almost surely as
  $t\rightarrow\infty$.
\end{remark}

\begin{lemma}[Finite first moments of \boldmath$\sigma$]\label{sigma_bounded}
  Consider a thinned Model~I, started with $\widehat X_0=x\geq 0$ and
  \begin{align}\label{sigma_stop}
    \widehat\sigma_0:=\inf\{t \ge 0:\widehat X_t = \widehat R_t\},
  \end{align}
  that is, $\widehat\sigma_0$ is the first time when then broken Brownian
  ratchet hits the moving reflection boundary. Then, using $\mathbb
  E_x[\cdot] := \mathbb E[\cdot|X_0=x]$,
  \begin{align*}
    \sup_{x}{\mathbb{E}_x[\widehat\sigma_0]}<\infty.
  \end{align*}
\end{lemma}

\begin{proof}
  Set $ H_t := \widehat X_t - \widehat R_t$ and note that $\mathcal H
  = ( H_t)_{t\geq 0}$ locally behaves like Brownian motion, and has
  jump discontinuities which, by time $t$, either increase or decrease $ 
  H_{t-}$. First, $H_{t-}$ jumps to $h<H_{t-}$ at rate $\gamma dh$ by
  occurrence of a new Poisson point in the graphical construction at
  rate $\gamma$. Second, it jumps to some $h>H_{t-}$ at rate $\delta$
  since $t=\tau+z$ for the active Poisson point $(\tau,r,z)$ by time
  $t-$.  Note that $\widehat\sigma_0 := \inf\{t\geq 0: H_t=0\}$. We
  couple the process $( H_t)_{t\geq 0}$ to a process $(I_t)_{t\geq 0}$
  with $H_t\leq I_t$ for all $0\leq t\leq \widehat\sigma_0$, define
  $\sigma_I:= \inf\{t\geq 0: I_t=0\}$ and show that $\sup_{x}\mathbb
  E_x[\sigma_I]<\infty$.
  
  The process $\mathcal I = (I_t)_{t\geq 0}$ has the following
  dynamics: it behaves locally like the same Brownian motion as
  $\mathcal H$, but starts in $I_0=\infty$. It jumps to 1 at time $t$
  in the following cases: (i) if $\mathcal H$ jumps from $ H_{t-}>1$
  to $H_t\leq 1$ (ii) if $\mathcal H$ jumps from $ H_{t-}<1$ to $
  H_t\leq H_{t-}$ (iii) at an independent rate $\gamma(1- H_t)^+$. In
  total this gives a jump rate of $\gamma$ to jump to 1 and $ H_t\leq
  I_t$ at such a jump time. In addition, $\mathcal I$ jumps to
  $\infty$ if $\mathcal H$ jumps from $ H_{t-}$ to $ H_t> H_{t-}$,
  which occurs at rate $\delta$.
  
  Consider an exponentially distributed random variable $T$ with rate
  $\gamma+\delta$ and a Brownian motion, starting in 1 at time $0$.
  Let $\sigma_B$ be its hitting time of 0 and set $p := \mathbb 
  P(\sigma_B \leq T)>0$. Note that $\sigma_I$ is given as follows: A
  Poisson point at rate $\gamma$ has to occur at time $s$ with the
  property that the Brownian motion, starting in 1 at time $s$ hits 0
  before any of the other events, bringing $\mathcal I$ back to 1 or
  infinity, occurs. We obtain
  \[\mathbb E_x[\sigma_I] = \frac{1}{p\gamma} + 
  \mathbb E[\sigma_B|\sigma_B<T] = \frac{1}{p\gamma} + \frac{\mathbb
    E[\sigma_B, \sigma_B<T]}{p} \leq \frac{1}{p} \Big(\frac{1}{\gamma}
  + \frac{1}{\gamma+\delta}\Big)<\infty, \] independently of $x$. The
  result follows since $\widehat\sigma_0\leq\sigma_I$, almost
  surely. \qed
\end{proof}

\begin{lemma}[Increments of \boldmath$(\widehat X_t)_{t\geq 0}$]\label{bounds}
  For the $\gamma/\delta$-broken Brownian ratchet, thinned Model~I,
  started in $x\geq 0$, we have
  \begin{setlength}{\leftmargini}{1cm} 
    \begin{enumerate}
      \renewcommand{\labelenumi}{(\roman{enumi})}
    \item $\mathbb{E}_x[\widehat X^2_{\widehat\sigma_0}]<\infty$,
    \item $0<\mathbb{E}_x[\widehat X_{\widehat\sigma_1}-\widehat
      X_{\widehat\sigma_0}]<\infty$,
    \item $0<\mathbb{E}_x[\widehat\sigma_1 - \widehat\sigma_0]<\infty$.
    \end{enumerate}
  \end{setlength}
  \end{lemma}
\begin{proof}
  For (i), note that $|\widehat X_{\widehat\sigma_0}|\stackrel d =
  |B_{\rho_0}|$. Using the second Wald identity and
  Lemma~\ref{sigma_bounded} we obtain
  \[ \mathbb{E}_x[\widehat
  X^2_{\widehat\sigma_0}]=\mathbb{E}_x[B^2_{\widehat\sigma_0}]
  =\mathbb{E}_x[\widehat\sigma_0]<\infty,\] which shows (i). For (ii)
  and (iii), we can assume that $x=0$ and $\widehat\sigma_0=0$ without
  loss of generality. Consider the time $\widehat t$ of the first jump
  of $(\widehat R_t)_{t\geq 0}$. Of course, we have $\widehat
  X_{\widehat t}>0$ \textit{almost surely} and therefore
  $0<\mathbb{E}[\widehat X_{\widehat t}]<\infty$. Moreover, we have
  that $(\widehat X_{\widehat\sigma_1 \wedge t}-\widehat X_{\widehat
    t})_{t\geq\widehat t}$ is a uniformly integrable
  martingale. Indeed, again using the second Wald identity and
  Lemma~\ref{sigma_bounded},
  \[ \sup_{t\geq \widehat t} \mathbb E_x[(\widehat X_{\widehat\sigma_1 \wedge
    t}-\widehat X_{\widehat t})^2] = \sup_{t\geq \widehat t} \mathbb
  E_x[\widehat\sigma_1 \wedge t - \widehat t\; ] = \mathbb E_x[\widehat\sigma_1 -
  \widehat t\; ] < \infty.\] 
  By the Optional Stopping Theorem we obtain that 
  \[ \mathbb E_0[\widehat X_{\widehat\sigma_1}] = \mathbb E_0[\widehat X_{\widehat t}]
  + \mathbb E_0[\widehat X_{\widehat\sigma_1} -  \widehat X_{\widehat t}] = \mathbb
  E_0[\widehat X_{\widehat t}] \] 
  and (ii) follows since $0<\mathbb{E}[\widehat X_{\widehat t}]<\infty$.  For
  (iii), rewrite $\widehat\sigma_1 = (\widehat\sigma_1-\widehat t)+\widehat t$. We
  know that $\widehat t$ is bounded from above by the killing time
  $\tau$ of a Brownian motion, if it is killed at rate $\gamma|B_t|$
  at time $t$. In \citet[Lemma 5.4]{depper_pfaffel:2010} it was shown
  that $\mathbb{E}_0[\tau]<\infty$. Hence, by 
  Lemma~\ref{sigma_bounded},
  \begin{align*}
    \mathbb E_0[\widehat\sigma_1] & = \mathbb E_0[\widehat\sigma_1-\widehat t]  
    + \mathbb E_0[\, \widehat t\; ] \\ & = \mathbb E_0 \left[\mathbb E_{\widehat
        X_{\widehat t} - \widehat R_{\widehat t}} [\widehat\sigma_0] \right]  +
    \mathbb E_0[\, \widehat t\; ] \leq \sup_{x\geq 0} \mathbb E_x[\widehat\sigma_0] +
    \mathbb E_0[\tau] < \infty, 
  \end{align*}
  which finishes the proof. \qed
\end{proof}

\begin{proposition}\label{remainder}
  Recall from Remark~\ref{remark_LLN} the $\gamma/\delta$-broken
  Brownian ratchet as a cumulative process with remainder
  $(A_t)_{t\geq 0}$ from \eqref{cumul_process}. Then,
  \begin{align*}
    \lim_{t\rightarrow\infty}\frac{A_t}{t}=0\quad\text{ almost
      surely.}
  \end{align*}
\end{proposition}

\begin{proof}
  We proceed as in Lemma 8 in \cite{Smith:1955}. Define for
  $n=1,2,\dots$
  \begin{align*}
    Y_n=\sup_{t\in[\widehat\sigma_{n-1},\widehat\sigma_n]}|\widehat X_t - \widehat
    X_{\widehat\sigma_n}|
  \end{align*}
  and note that $Y_1, Y_2,\dots$ are independent and identically
  distributed. We write, using Lemma~\ref{bounds}(i) and the law of
  large numbers for $(M_t)_{t\geq 0}$
  \begin{equation}
    \label{eq:bounds1}
    \begin{aligned}
      \lim_{t\to\infty} \frac{|A_t|}{t} & \leq \lim_{t\to\infty}
      \frac{|\widehat X_{\widehat\sigma_0}|}{t} + \lim_{t\to\infty}
      \frac{|\widehat X_t - \widehat X_{\widehat\sigma_{M_t}}|}{t} \leq
      \lim_{t\to\infty} \frac{M_t}{t} \cdot \lim_{t\to\infty}
      \frac{Y_{M_t}}{M_t} \\ & = \frac{1}{\mathbb E[\widehat\sigma_1 -
        \widehat\sigma_0]} \cdot \lim_{n\to\infty} \frac{Y_{n}}{n}.
    \end{aligned}
  \end{equation}
  Since $Y_1\leq \widehat X_{\widehat\sigma_1} + \sup_{t\geq 0}
  \widehat X_{t\wedge \widehat\sigma_1}$, by Doob's inequality,
  \begin{align*}
    \mathbb E[Y_1^2] \leq 2\mathbb E[\widehat X_{\widehat\sigma_1}^2] +
    2\mathbb E[\sup_{t\geq 0} \widehat X_{t\wedge \widehat\sigma_1}^2] \leq
    2\mathbb E[\widehat X_{\widehat\sigma_1}^2] + 8 \mathbb E[\widehat
    X_{\widehat\sigma_1}^2] < \infty,
  \end{align*}
  it follows that $Y_n/n \to 0$ almost surely, as $n\to\infty$ by
  Lemma~7 in \cite{Smith:1955}. Now, the result is implied
  by~\eqref{eq:bounds1}. \qed
\end{proof}

\subsection{Proof of Theorem~\ref{T:speedI}}
\label{sub:proofI} The proof consists of three steps. We derive the
lower and upper bound for the speed of the broken ratchet. Then, we
prove the scaling property.

~

\noindent {\bf Step 1: Lower bound for the speed}: The thinned Model~I
of the $\gamma/\delta$-broken Brownian ratchet uses less possible
reflection boundaries, leading to a smaller speed. Hence, if we can
show that it still has positive speed we obtain that Model~I has
positive speed as well, as claimed in Theorem~\ref{T:speedI}.

In order to prove that thinned Model~I has a positive speed, we use
the law of large numbers for the cumulative process as introduced in
Remark~\ref{remark_LLN}. By the law of large numbers for
$(S_n)_{n=1,2,\dots}$ and $(M_t)_{t\geq 0}$, Lemma~\ref{bounds} and
Proposition~\ref{remainder}, we have
\begin{align*}
  \frac{\widehat X_t}{t} & = \frac{S_{M_t}}{M_t} \frac{M_t}{t} +
  \frac{A_t}{t} \xrightarrow{t\to\infty}\frac{\mathbb{E}[\widehat
    X_{\widehat\sigma_1} - \widehat X_{\widehat\sigma_0}]}{\mathbb
    E[\widehat\sigma_1 - \widehat\sigma_0]}=: c \text{ almost surely.}
\end{align*}

~

\noindent {\bf Step 2: Upper bound for the speed}: The speed of the
$\gamma/\delta$-broken Brownian ratchet is decreasing in $\delta$. In
particular, taking $\delta=0$ gives an upper bound for the
speed. Moreover this bound is independent of $\delta$. We find from
\cite{depper_pfaffel:2010} that
\[\frac{X_t^{\gamma,\delta}}{t}\leq\frac{X_t^{\gamma,0}}{t}
\xrightarrow{t\to\infty}C \text{ almost surely.}\] for
some $0<C<\infty$.

~

\noindent {\bf Step 3: Scaling property}: As a consequence of
Proposition~\ref{P:scaleI} we have
\begin{align*}
  \liminf_{t\rightarrow\infty}\frac{\mathbb{E}[X_t^{\gamma;\delta}]}{t}
  & =\liminf_{t\rightarrow\infty}\gamma^{-\frac{1}{3}}\frac{\mathbb{E}
    [X_{\gamma^{2/3}t}^{1;\delta/\gamma^{2/3}}]}{t}
  =\liminf_{t\rightarrow\infty}\gamma^{\frac{1}{3}}
  \frac{\mathbb{E}[X_{\gamma^{2/3}t}^{1;\delta/\gamma^{2/3}}]}{\gamma^{2/3}t}
  \\ & =\liminf_{t\rightarrow\infty}\gamma^{\frac{1}{3}}
  \frac{\mathbb{E}[X_{t}^{1;\delta/\gamma^{2/3}}]}{t}
\end{align*}
where $\liminf$ can as well be replaced by $\limsup$.

\section{Proofs: Model~II}
We start in Section~\ref{sub:scaleII}, Proposition~\ref{P:scaleII},
with showing the same scaling property as given in
Proposition~\ref{P:scaleI} for Model~I. The rest of the section is
concerned with concrete calculations. Model~II is linked to killed
Brownian motion in Section~\ref{sub:killed}. We interpret Model~II as
a (random) additive functional of a Markov chain, where the related Markov
chain is defined in Definition~\ref{Y_W_eta} and is 
shown to have a unique equilibrium in Propositions~\ref{P4} and~\ref{P5} in
Section~\ref{ss:markov}. For this equilibrium, it is possible to compute
increments analytically, which is carried out in the proof of
Theorem~\ref{T:speedII} in Section~\ref{sub:proofT2}.  

\subsection{Scaling property}
\label{sub:scaleII}
The following result is the analogous result of
Proposition~\ref{P:scaleI} for Model~II.

\begin{proposition}[Scaling property, Model~II]\label{P:scaleII}
  Let $(\widetilde X_t^{\gamma;\delta}, \widetilde
  {R}_t^{\gamma;\delta})_{t\geq 0}$ be the $\gamma/\delta$-broken
  Brownian ratchet, Model~II, with $\gamma>0, \delta\geq 0$ and
  initial value $(0,0)$. Then, the following holds:
  \begin{align}\label{skalierung2}
    (\widetilde X_t^{\gamma;\delta}, \widetilde
    R_t^{\gamma;\delta})_{t\geq 0}\overset{d}{=}
    \gamma^{-\frac{1}{3}}\left(\widetilde
      X_{\gamma^{2/3}t}^{1;\delta/\gamma^{2/3}}, \widetilde
      R_{\gamma^{2/3}t}^{1;\delta/\gamma^{2/3}}\right)_{t\geq 0}.
  \end{align}
\end{proposition}

\begin{proof}
  We use supercripts $\gamma,\delta$ in order to emphasize dependency
  on the parameters and start with the $(\widetilde
  X_t^{\gamma,\delta}, \widetilde R_t^{\gamma,\delta})_{t\geq
    0}$-ratchet.  Recall the notation from the proof of
  Proposition~\ref{P:scaleI} as well as $\widetilde
  g(N^\gamma)\stackrel d = N^1$. Note that the rate-$\delta$ random
  variables $Z^\delta_1, Z^\delta_2,\dots$ in Definition~\ref{modelII}
  only affect the $t$-direction, while the random variables
  $E_1^{\gamma;\delta}, E_2^{\gamma;\delta}, E_3^{\gamma;\delta},
  \dots$ at rate $\gamma/\delta$ affect the $x$-direction. Since,
  $\widetilde g_\gamma(Z^\delta_n, E_n^{\gamma;\delta}) \stackrel d =
  (Z_n^{\delta/\gamma^{2/3}}, E_n^{1;\delta/\gamma^{2/3}})$, and since
  the same rescaling as in \eqref{eq:BB1} holds in Model~II, we have
  shown that
  \[  \gamma^{1/3}(\widetilde X_t^{\gamma;\delta},\widetilde
  R_t^{\gamma;\delta})_{t\geq 0} \stackrel d = (\widetilde X_{\gamma^{2/3} 
    t}^{1;\delta/\gamma^{2/3}},\widetilde R_{\gamma^{2/3}
    t}^{1;\delta/\gamma^{2/3}})_{t\geq 0}\]
  and the assertion follows. \qed
\end{proof}

\subsection{Connections to killed Brownian motion}
\label{sub:killed}
The Brownian motion which drives the broken Brownian ratchet,
Model~II, is reflected at the same boundary until $R_t$ changes. We
interpret the time when this happens as a killing time of the Brownian
motion. Due to the scaling property of the Brownian ratchet it would
be possible to carry out the computations in the notationally
convenient case $\gamma=\frac{1}{2}$ and deduce the general result
from this particular one. However, we feel that an approach explicitly
allowing for all values of $\gamma$ is more transparent.

\begin{definition}[Killed reflected Brownian motion]
  \label{def:killedBM}
  Let $(B_t)_{t\geq 0}$ denote Brownian motion started in
  $x\geq 0$ and consider the stopping time $\eta$ defined by
  \begin{align}
    \label{eq:killingtime}
    \eta:=\inf\Big\{t >0 : \int_0^t (\gamma |B_s| +\delta)\, ds \ge
    \xi\Big\},
  \end{align}
  where $\xi$ is an independent exponentially distributed random
  variable of rate~$1$.
  
  Define $\widetilde{\mathcal{B}}:=(\widetilde B_t)_{t\geq 0}$ by
  $\widetilde B_t:=|B_t|$ for $0\leq t< \eta$ and $\widetilde
  B_t=\Delta$ for $t\geq \eta$, where $\Delta\not\in\mathbb{R}$ is the
  cemetery state. Then $\widetilde{\mathcal{B}}^x =
  \widetilde{\mathcal B}$ is reflected Brownian motion killed at rate
  $\gamma \widetilde B_t + \delta$ at time $t$. Denote the
  probability measure of the Brownian motion started in $x$ by
  $\mathbb{P}_x$ and write $\mathbb{E}_x$ for the respective
  expectation.
\end{definition}

\noindent
The second order operator $\mathcal A$ associated to the killed
Brownian motion defined above acts on $C^2$ functions $f: [0,\infty)
\to \mathbb R$ satisfying $f'(0+)=0$ according to
\begin{align}
  \label{eq:killed-operator}
   \mathcal A f(x) = \frac12 f''(x) -(\gamma x +\delta) f(x).  
\end{align}
The diffusion process $\widetilde{\mathcal{B}}$ is transient since it
can be killed in each interval with positive probability. Let
$p(\cdot;\cdot,\cdot)$ denote the corresponding transition density
with respect to the speed measure which is given by $m(dx)=2dx$ (see
\citet[p.~17]{boro_sal:2002}). In the same way as in the proof of
Lemma~5.4. in \cite{depper_pfaffel:2010}, the Green function of
$\widetilde{\mathcal B}^x$ is given by
\begin{align*}
  G(x,y):=\int_0^\infty p(t;x,y)\, dt,
\end{align*}
which is finite for each $x,y\geq 0$. It can be written in terms of
solutions of the differential equation $\mathcal A u=0$ which reads
\begin{align}
  u''(x)-(2\gamma x + 2\delta)u(x) =0\label{2nd_operator2}.
\end{align}
The following lemma is the basis for all calculations to follow.

\begin{lemma}[Green function of killed Brownian motion]
  The Green function of the killed Brownian motion from
  Definition~\ref{def:killedBM} is given by
  \begin{align}\label{Green_new}
    G(x,y):= \frac1w 
    \begin{cases}
      \psi(x)\phi(y)	& \text{for } 0\leq x \leq y,\\ 
      \psi(y)\phi(x)	& \text{for } 0\leq y \leq x, 
    \end{cases}
  \end{align}
  where 
  \begin{align}
    \label{eq:w} 
    w & = \frac{(2\gamma)^{1/3}}\pi, \intertext{and}
    \label{eq:phi}
    \phi(x) & = Ai\Big( (2\gamma)^{1/3} x +
    \tfrac{2^{1/3}\delta}{\gamma^{2/3}}\Big), \\
    \label{eq:psi} 
    \psi(x) & = Bi\Big( (2\gamma)^{1/3} x +
    \tfrac{2^{1/3}\delta}{\gamma^{2/3}}\Big) - C Ai\Big(
    (2\gamma)^{1/3} x + \tfrac{2^{1/3}\delta}{\gamma^{2/3}}\Big),
    \intertext{with} C & =C(\gamma,\delta)=
    \frac{Bi'(\tfrac{2^{1/3}\delta}{\gamma^{2/3}})}{
      Ai'(\tfrac{2^{1/3}\delta}{\gamma^{2/3}})}
  \end{align}
  are two solutions of~\eqref{2nd_operator2}.
\end{lemma}
\begin{proof}
  In order to compute the Green function of the killed Brownian motion
  we need two particular solutions, say $f$ and $g$, of the equation
  \eqref{2nd_operator2} on $[0,\infty)$ such that \citep[see
  e.g.][p.~18--19]{boro_sal:2002}
  \begin{setlength}{\leftmargini}{1cm} 
    \begin{enumerate}
    \item[(i)] $f$ is positive, strictly decreasing and $f(x) \to 0$ as $x 
      \to \infty$,
    \item[(ii)] $g$ is positive, strictly increasing, 
    \item[(iii)] $g'(0)=0$ (this is the condition for reflecting boundary).  
    \end{enumerate}
  \end{setlength}
  The functions $x\mapsto Ai\Big( (2\gamma)^{1/3} x +
  \tfrac{2^{1/3}\delta}{\gamma^{2/3}}\Big)$ and $x\mapsto Bi\Big(
  (2\gamma)^{1/3} x + \tfrac{2^{1/3}\delta}{\gamma^{2/3}}\Big)$ are
  two linearly independent solutions of \eqref{2nd_operator2}. It is
  easy to check that the requirements (i)--(iii) are satisfied by
  functions $f=\phi$ and $g=\psi$ defined in \eqref{eq:phi}
  respectively \eqref{eq:psi}.
  
  The Wronskian of $\psi$ and $\phi$ is given by
  \begin{equation}
    \label{eq:wronsk}
    \begin{aligned}
      w & =\psi'(0)\phi(0)-\psi(0)\phi'(0) = -\psi(0)\phi'(0) \\ & =
      (2\gamma)^{1/3}\big(Bi'(\tfrac{2^{1/3}\delta}{\gamma^{2/3}})
      Ai(\tfrac{2^{1/3}\delta}{\gamma^{2/3}})  -
      Bi(\tfrac{2^{1/3}\delta}{\gamma^{2/3}})
      Ai'(\tfrac{2^{1/3}\delta}{\gamma^{2/3}})\big) =
      \frac{(2\gamma)^{1/3}}{\pi}, 
    \end{aligned}
  \end{equation}
  hence all results follow from \citep[see
  e.g.][p.~18-19]{boro_sal:2002}. \qed 
\end{proof}

\noindent
The density of the killing position with respect to Lebesgue measure
is $k(y)=2\gamma y+2\delta$ and moreover, the killing position of the
Brownian motion started in $x$ has density
\begin{align}\label{fy}
  f(y):=G(x,y)k(y)
\end{align}
with respect to Lebesgue measure \citep[see][p.~17 resp. p.~14]{boro_sal:2002}. 

\noindent
The next lemma will be used in the proof of Proposition~\ref{P4}.

\begin{lemma}[Bounds on increments of killed Brownian motion]
  For any $x \ge 0$
  \begin{align}
    \label{eq:1stmomkilledB}
    \frac{\phi(0)\psi(0)}w \le \mathbb E_x[\widetilde B_{\eta-}] \le x +
    \frac{\phi(0)\psi(0)}w.   
  \end{align}
\end{lemma}
\begin{proof}
  Using the fact that $\psi$ and $\phi$ solve \eqref{2nd_operator2} and then
  integration by parts we obtain 
  \begin{align*}
    \mathbb E_x[\widetilde B_{\eta-}] & = \int_0^\infty y G(x,y)k(y)
    \, dy \\ & = \frac{\psi(x)}w \int_x^\infty y \phi(y) (2\gamma y+
    2\delta)\, dy + \frac{\phi(x)}w \int_0^x y \psi(y) (2\gamma y +
    2\delta)\, dy \\ & = \frac{\psi(x)}w \int_x^\infty y \phi''(y) \,
    dy + \frac{\phi(x)}w \int_0^x y \psi''(y)\, dy \\ & = \frac{x}w(-
    \psi(x) \phi'(y) + \phi(x)\psi'(x)) + \frac{\phi(x)\psi(0)}w = x +
    \frac{\phi(x)\psi(0)}w.
  \end{align*}
  Now using the fact that $\phi$ is strictly decreasing we obtain the
  upper bound in \eqref{eq:1stmomkilledB}. Furthermore, due to
  $\psi'(0)=0$ by the choice of $C$ we have $\psi(0)=-w/\phi'(0)=
  -w/Ai'(\tfrac{2^{1/3}\delta}{\gamma^{2/3}})$. As the function
  \begin{align*}
    x \mapsto x + \frac{\phi(x)\psi(0)}w = x - \frac{Ai\Big(
      (2\gamma)^{1/3} x +
      \tfrac{2^{1/3}\delta}{\gamma^{2/3}}\Big)}{Ai'(\tfrac{2^{1/3}\delta}{\gamma^{2/3}})}
  \end{align*}
  is strictly increasing  on $[0,\infty)$ its minimum is attained in
  $x=0$. This gives the lower bound in \eqref{eq:1stmomkilledB}. 
  \qed 
\end{proof}

\subsection{The invariant distribution for increments at jump
  times} \label{ss:markov} One advantage of Model~II as compared to
Model~I is that we can define a Markov chain which models time- and
space-increments at jump times of the reflection boundary. We show
that the Markov chain has an equilibrium (Proposition~\ref{P4}) which
is unique (Proposition~\ref{P5}).

\begin{definition}[Markov chain at jump times]\label{Y_W_eta}
  Consider the $\gamma/\delta$-broken Brownian ratchet $(\widetilde
  X_t, \widetilde R_t)_{t \ge 0}$ from Definition~\ref{modelII} with
  $\widetilde X_0=x$ and $\widetilde R_0=0$. Define
  $(Y_n,W_n,\eta_n,)_{n\geq 1}$ by
  \begin{align}
    Y_n:=\widetilde X_{\tilde t_n}-\tilde r_{n},\quad 
    W_n:=\tilde r_{n}-\tilde r_{n-1} \quad \text{and} \quad 
    \eta_n:=\tilde t_n-\tilde t_{n-1}. 
  \end{align}
  Note that for any $k$, $(Y_n,W_n,\eta_n)_{n=k+1,k+2,\hdots}$ depends on 
  $(Y_n,W_n,\eta_n)_{n=1,2,\hdots,k}$ only through $Y_k$. That is,
  $(Y_n,W_n,\eta_n,)_{n\geq 1}$ is a Markov chain. For any $n \ge 1$ we denote
  the law of  $(Y_n, W_n, \eta_n)_{n \ge 1}$  by $(P^n_x)$.  
\end{definition}

\begin{remark}[Distribution of \boldmath$Y_n$ and $W_n$]
  \label{rem:distrWY}
  Let $(E_n)_{n=0,1,\dots }$ be iid exponentially distributed with
  parameter $\gamma/\delta$ and let $(U_n)_{n=0,1,\dots }$ be iid
  uniformly distributed on $[0,1]$. Then, recalling
  Definition~\ref{modelII},
  \begin{align*}
    Y_{n+1}  &  \stackrel d=
    \begin{cases}
      (\widetilde X_{\tilde t_{n+1}}-\tilde r_n)(1 - U_n), & \text{ if
      } \tilde r_{n+1} > \tilde r_n,\\ (\widetilde X_{\tilde
        t_{n+1}}-\tilde r_n + E_n), & \text { if } \tilde r_{n+1} \le
      \tilde r_n,
    \end{cases} \\ \intertext{and} W_{n+1} & \stackrel d=
    \begin{cases}
      (\widetilde X_{\tilde t_{n+1}}-\tilde r_n)U_n, & \text{ if }
      \tilde r_{n+1} > \tilde r_n,\\ - E_n, & \text { if } \tilde
      r_{n+1} \le \tilde r_n
    \end{cases}
  \end{align*}
  where $\stackrel d=$ means equality in distribution. In particular,
  on the event $\{ \tilde r_{n+1} > \tilde r_n\}$ the random variables
  $W_{n+1}$ and $Y_{n+1}$ have the same distribution.
\end{remark}

\noindent
In the following two propositions we obtain existence and uniqueness
of the invariant distribution of the Markov chain
$(Y_n,W_n,\eta_n)_{n\ge 1}$.

\begin{proposition}[Existence of invariant distribution]\label{P4}
  For any $x\geq 0$ the sequence of Ces\`aro averages 
  $\left(n^{-1}\sum_{k=1}^{n}{P^k_x}\right)_{n \ge 1}$ converges along a
  subsequence weakly to an invariant distribution $P_x$ of the Markov chain
  $(Y_n,W_n,\eta_n)_{n\ge 1}$.  
\end{proposition}
\begin{proof} 
  In the case $\delta =0$ the assertion is shown in Proposition~5.6 in
  \cite{depper_pfaffel:2010}.  Thus, we assume $\delta>0$ in the rest
  of the proof. We only need to show that the first moments of $Y_n$,
  $W_n$ and $\eta_n$ are bounded uniformly in $n$. This implies that
  the sequence $(P_x^n)_{n \ge 1}$ is tight. Then the sequence of Ces\`aro
  averages $\left(n^{-1}\sum_{k=1}^{n}{P^k_x}\right)_{n \ge 1}$ is also tight. That is,
  any subsequence contains another subsequence along which the Ces\`aro
  averages converge weakly. These weak limits are invariant for the Markov
  chain $(Y_n,W_n,\eta_n)_{n\ge 1}$. A proof of this fact in the continuous
  time case, that can be easily adapted to the discrete time case, can
  be found in \citet[p.~11]{Liggett:85}. 

  The moments of each $\eta_n$ are bounded from above, since by
  construction, see \eqref{eq:taun}, each $\eta_n$ is bounded by an
  exponential random variable with parameter $\delta$. Furthermore by
  construction we have (see Remark~\ref{rem:distrWY})
  \begin{align*}
    Y_{n+1} & = (\widetilde X_{\tilde t_{n+1}}-\tilde
    r_n)\big(\ind_{\{\tilde r_{n+1} > \tilde r_n\}} (1-U_n) +
    \ind_{\{\tilde r_{n+1} \le \tilde
      r_n\}}\big) + \ind_{\{\tilde r_{n+1} \le \tilde r_n\}}  E_n \\
    & \le (\widetilde X_{\tilde t_{n+1}}-\tilde r_n)(1-\ind_{\{\tilde
      r_{n+1} > \tilde r_n\}} U_n) + E_n.
  \end{align*}
  We set $\mathcal{G}_n = \sigma((\widetilde X_{t},\widetilde
  R_{t})_{0 \le t \le\tilde t_n})$ and $c:=\phi(0)\psi(0)/w$. By the
  strong Markov property we obtain
  \begin{align*}
    \mathbb E [(\widetilde X_{\tilde t_{n+1}}-\tilde r_n)| \mathcal G_n]  =
    \mathbb E_{Y_n} [\widetilde X_{\tilde  t_{1}}-\tilde r_0] =   \mathbb
    E_{Y_n} [\widetilde B_{\eta-}], %   \le Y_n + c,
  \end{align*} 
  where $\widetilde B_t$ is killed reflecting Brownian motion from 
  Definition~\ref{def:killedBM} and $\eta$ is
  its killing time. Furthermore  
  \begin{align*}
    \mathbb E [(\widetilde X_{\tilde t_{n+1}}-\tilde r_n)\ind_{\{\tilde r_{n+1}>
      \tilde r_n\}} | \mathcal G_n] =  \mathbb E_{Y_n} [(\widetilde X_{\tilde
      t_{1}}-\tilde r_0)\ind_{\{\tilde r_{1}> \tilde r_0\}}] = \mathbb E_{Y_n} 
    \left[ \frac{\widetilde B_{\eta-}^2}{\widetilde
        B_{\eta-}+\delta/\gamma}\right]. 
  \end{align*}
  The function $x \mapsto x^2/(x+\delta/\gamma)$ is convex. Thus, by Jensen's
  inequality 
  \begin{align*}
    \mathbb E [(\widetilde X_{\tilde t_{n+1}}-\tilde r_n)\ind_{\{\tilde r_{n+1}> 
      \tilde r_n\}} | \mathcal G_n] \ge \mathbb E_{Y_n}[\widetilde B_{\eta-}] 
    \frac{\mathbb E_{Y_n}[\widetilde B_{\eta-}]}{\mathbb E_{Y_n}[\widetilde
      B_{\eta-}]+\delta/\gamma}.    
  \end{align*}
  Now using that the function $x\mapsto x/(x+\delta/\gamma)$ is increasing and
  \eqref{eq:1stmomkilledB} we obtain 
  \begin{align*}
    \mathbb E [(\widetilde X_{\tilde t_{n+1}}-\tilde r_n)\ind_{\{\tilde r_{n+1}> 
      \tilde r_n\}} | \mathcal G_n] \ge \mathbb E_{Y_n}[\widetilde B_{\eta-}] 
    \frac{c}{c + \delta/\gamma}.  
  \end{align*}
  Altogether, setting
  \begin{align*}
    q:= 1-\frac12 \frac{c}{c + \delta/\gamma} \in (0,1) 
  \end{align*}
  we obtain 
  \begin{align*}
    \mathbb E_x[ Y_{n+1}] & \le \mathbb E_x [(\widetilde X_{\tilde
      t_{n+1}}-\tilde r_n)(1-\ind_{\{\tilde r_{n+1} > \tilde r_n\}}
    U_n) + E_n] \\ & = \mathbb E_x \left[ \mathbb E[(\widetilde
      X_{\tilde
        t_{n+1}}-\tilde r_n)(1-\ind_{\{\tilde r_{n+1} > \tilde r_n\}} U_n)|
      \mathcal G_n]\right] +  \delta/\gamma   \\ 
    & \le q \mathbb E_x\left[ \mathbb E_{Y_n} [\widetilde
      B_{\eta-}]\right] +
    \delta/\gamma \\
    & \le q (\mathbb E_x[Y_n ] + c) + \delta/\gamma,
  \end{align*}
  where we used \eqref{eq:1stmomkilledB} for the last inequality.
  Uniform boundedness of the first moments of $Y_n$ follows since the
  recursion $x_{n+1} = qx_n + d$ has a unique fixed point for $|q|<1$
  and any $d$.  For $W_{n+1}$ we have
  \begin{align*}
    W_{n+1} & = \ind_{\{\tilde r_{n+1} > \tilde r_n\}} U_n(\widetilde
    X_{\tilde t_{n+1}}-\tilde r_n) - \ind_{\{\tilde r_{n+1} \le \tilde
      r_n\}} E_n.
  \end{align*}
  Now using that on $\{\tilde r_{n+1} > \tilde r_n\}$, $W_{n+1}=|W_{n+1}|$ and
  $Y_{n+1}$ have the same distribution we obtain  
  \begin{align*}
    \mathbb E_x[|W_{n+1}|] = \mathbb E_x[\ind_{\{\tilde r_{n+1} >
      \tilde r_n\}} Y_{n+1} + \ind_{\{\tilde r_{n+1} \leq \tilde
      r_n\}} E_n] \leq \mathbb E_x[Y_{n+1}] + \delta/\gamma.
  \end{align*}
  Uniform boundedness of first moments of $W_n$ follows from uniform
  boundedness of first moments of $Y_n$. \qed
\end{proof}

\begin{proposition}[Uniqueness of invariant distribution]\label{P5}
  Consider the Markov chain $(Y_n,W_n,\eta_n)_{n=1,2,\dots}$ from
  Definition~\ref{Y_W_eta} on the basis of a broken Brownian ratchet,
  Model~II started in $x\geq 0$. If an invariant distribution of the
  Markov chain exists, then it is unique.
\end{proposition}

\begin{proof}
  In the proof we use a second graphical construction for the
  $\gamma/\delta$-broken Brownian ratchet, Model~II, $(\widetilde X_t,
  \widetilde R_t)_{t\geq 0}$, which is based on a single Brownian
  motion $(B_t)_{t \ge 0}$ with $B_0=0$; see also
  Figure~\ref{fig:coupling}(A).

  We define the sequence of \emph{active points}
  $(S_n)_{n=0,1,\dots}$, the sequence of reflection boundaries
  $(R_n)_{n=0,1,\dots}$ and a sequence of jump times
  $(\kappa_{n})_{n=0,1,\dots}$ as follows: Let $(\xi_n)_{n=0,1,\dots}$
  be iid exponentially distributed with parameter $1$,
  $(E_n)_{n=0,1,\dots}$ iid exponentially distributed with parameter
  $\gamma/\delta$, and let $(U_n)_{n=0,1,\dots}$ be iid uniformly
  distributed on $(0,1)$. We define $ \kappa_0=0$, $S_0 = x$, $R_0=0$
  and for $n\ge 0$
  \begin{align*}
    \kappa_{n+1} &= \inf\Big\{t > \kappa_{n} : \int_{\kappa_n}^t
    (\gamma |B_s - S_{n}| + \delta) \, ds \ge \xi_{n}\Big\}
  \end{align*}
  is the time of the next event after $\kappa_n$ with rate
  $\gamma|B_s-S_n|+\delta$ at time $s$. At time $\kappa_{n+1}$, there
  are two possibilities:
  \begin{itemize}
  \item[$\bullet$] If $\{S_n \geq B_{\kappa_{n+1}}\}$, set
    \begin{align*}
      (S_{n+1},R_{n+1}) & = \begin{cases} (S_n -
        U_n(S_n - B_{\kappa_{n+1}}),R_{n}+U_n(S_n - B_{\kappa_{n+1}})) \\
        \phantom{ (S_n + E_n,R_n-E_n) } \text{ with probability
        }\frac{\gamma|B_{\kappa_{n+1}} - S_n|}{\gamma|B_{\kappa_{n+1}
          }- S_n| + \delta}, \\ (S_n + E_n,R_n-E_n) \text{ with
          probability }\frac{\delta}{\gamma|B_{\kappa_{n+1} }- S_n| +
          \delta}.\end{cases}
    \end{align*}
  \item[$\bullet$] If $\{S_n < B_{\kappa_{n+1}}\}$, set
    \begin{align*}
      (S_{n+1},R_{n+1}) & = \begin{cases} (S_n +
        U_n(B_{\kappa_{n+1}}-S_n
        ),R_{n}+U_n(B_{\kappa_{n+1}}-S_n)) \\
        \phantom{(S_n - E_n,R_n-E_n) } \text{ with probability
        }\frac{\gamma|B_{\kappa_{n+1}} - S_n|}{\gamma|B_{\kappa_{n+1}
          }- S_n| + \delta}, \\ (S_n - E_n,R_n-E_n) \text{ with
          probability }\frac{\delta}{\gamma|B_{\kappa_{n+1} }- S_n| +
          \delta}.\end{cases}
    \end{align*}
  \end{itemize}
  Thus, the active point jumps uniformly between the currently active
  point and the Brownian motion at rate $\gamma$ times their distance
  or it jumps an exponential distance away from the Brownian
  motion. In particular,
  \begin{align*}
    R_{n+1} & = \begin{cases} R_{n}+U_n|B_{\kappa_{n+1}}-S_n| &
      \text{ with probability }\frac{\gamma|B_{\kappa_{n+1}} -
        S_n|}{\gamma|B_{\kappa_{n+1} }- S_n| + \delta}, \\ R_n - E_n &
      \text{ with probability }\frac{\delta}{\gamma|B_{\kappa_{n+1} }-
        S_n| + \delta},\end{cases}
  \end{align*}
  i.e.\ the reflection boundaries either jump up at rate $\gamma$
  times the distance of the Brownian motion to the active point or
  down at rate $\delta$.

  Now set
  \begin{align*}
    \widetilde R_t & = R_n, \quad \widetilde S_t = S_n, \quad
    \text{for $t \in [\kappa_n,\kappa_{n+1})$}\\ \intertext{and}
    \widetilde X_t & =\widetilde R_t + |B_t-\widetilde S_t| \quad
    \text{for $t \ge 0$.}
  \end{align*}
  Then, $(\widetilde X_t)_{t\geq 0}$ is reflected at $(\widetilde
  R_t)_{t\geq 0}$ and is continuous, since for all $n\geq 1$,
  \begin{align*}
    \widetilde X_{\kappa_{n+1}-} & = R_n + |B_{\kappa_{n+1}-}-S_n|,\\
    \widetilde X_{\kappa_{n+1}} & = \begin{cases} R_n +
      U_n|B_{\kappa_{n+1}-}-S_n| + |B_{\kappa_{n+1}} - S_n -
      U_n(B_{\kappa_{n+1}}-S_n)| \\ \qquad \qquad \qquad \qquad \qquad
      \qquad \text{ with probability }\frac{\gamma|B_{\kappa_{n+1}} -
        S_n|}{\gamma|B_{\kappa_{n+1} }- S_n| + \delta},\\ R_n-E_n + |
      B_{\kappa_{n+1}} - S_n| + E_n \\ \qquad \qquad \qquad \qquad
      \qquad \qquad \text{ with probability
      }\frac{\delta}{\gamma|B_{\kappa_{n+1} }- S_n| + \delta}
    \end{cases}
    \\ & = R_n + |B_{\kappa_{n+1}}-S_n|.
  \end{align*}
  In other words, the process $(\widetilde X_t,\widetilde R_t)_{t\ge
    0}$ is a $\gamma/\delta$-broken Brownian ratchet, Model~II
  starting in $(|x|,0)$.

  We now prove Proposition~\ref{P5} using a coupling argument,
  illustrated in Figure~\ref{fig:coupling}(B). Recall that a coupling
  of two $\gamma/\delta$-broken Brownian ratchets, Model~II, is a
  process $(\widetilde X_t^{(1)}, \widetilde R_t^{(1)}, \widetilde
  X_t^{(2)}, \widetilde R_t^{(2)})$ such that $(\widetilde X_t^{(i)},
  \widetilde R_t^{(i)})$ is a $\gamma/\delta$-broken Brownian ratchet,
  $i=1,2$. It is straightforward, using the same Brownian motion, to
  construct such a coupling along the lines of the above construction,
  where $(\widetilde X_t^{(1)}, \widetilde R_t^{(1)})$ starts in
  $(x^{(1)},0)$ and $(\widetilde X_t^{(2)}, \widetilde R_t^{(2)})$ in
  $(x^{(2)},0)$, respectively, for $x^{(1)} , x^{(2)} \ge 0$. Then, we
  must show that both broken ratchets use the same active points after
  an almost surely finite time. This suffices, since from this time on
  the time- and space-increments at jump times are the same for both
  processes, and hence, both must converge to the same limit law \citep[see
  e.g.][Thm.~21.12]{Lindvall:1992}. 

  In order to show that both broken ratchets use the same active
  points after finite time, note the following: if the active points
  of both ratchets are both above or both below the Brownian motion
  they have a chance to jump to the same point (see
  Figure~\ref{fig:coupling}(B)). Assume for definiteness that at time
  $t$ we have $\widetilde S_t^{(2)} < \widetilde S_t^{(1)} < 
  B_t$ (as at time $t$ in Figure~\ref{fig:coupling}(B)). 
  Then at rate $\gamma (\widetilde S_t^{(2)} -
  \widetilde S_t^{(1)})$ only the lower active point jumps towards the
  $B_t$ and if it does, it jumps uniformly between $\widetilde
  S_t^{(1)}$ and $\widetilde S_t^{(2)}$. Furthermore at rate $\gamma
  (B_t-S_t^{(2)} )$ both jump to the same uniform position between
  $S_t^{(2)}$ and $B_t$. From then on the active points evolve
  together and we refer to this time as the coupling time ($T_c$ in 
  Figure~\ref{fig:coupling}(B)). After the 
  coupling time both ratchets have the same jump times and the same
  increments. Since the Brownian motion will spend infinite amount of
  time below or above both active points the coupling time is almost
  surely finite.  \qed
\end{proof}

\begin{figure}[htb]
  \begin{center}
    (A) \hspace{5cm} (B)\\
    \includegraphics[width=0.45\textwidth]{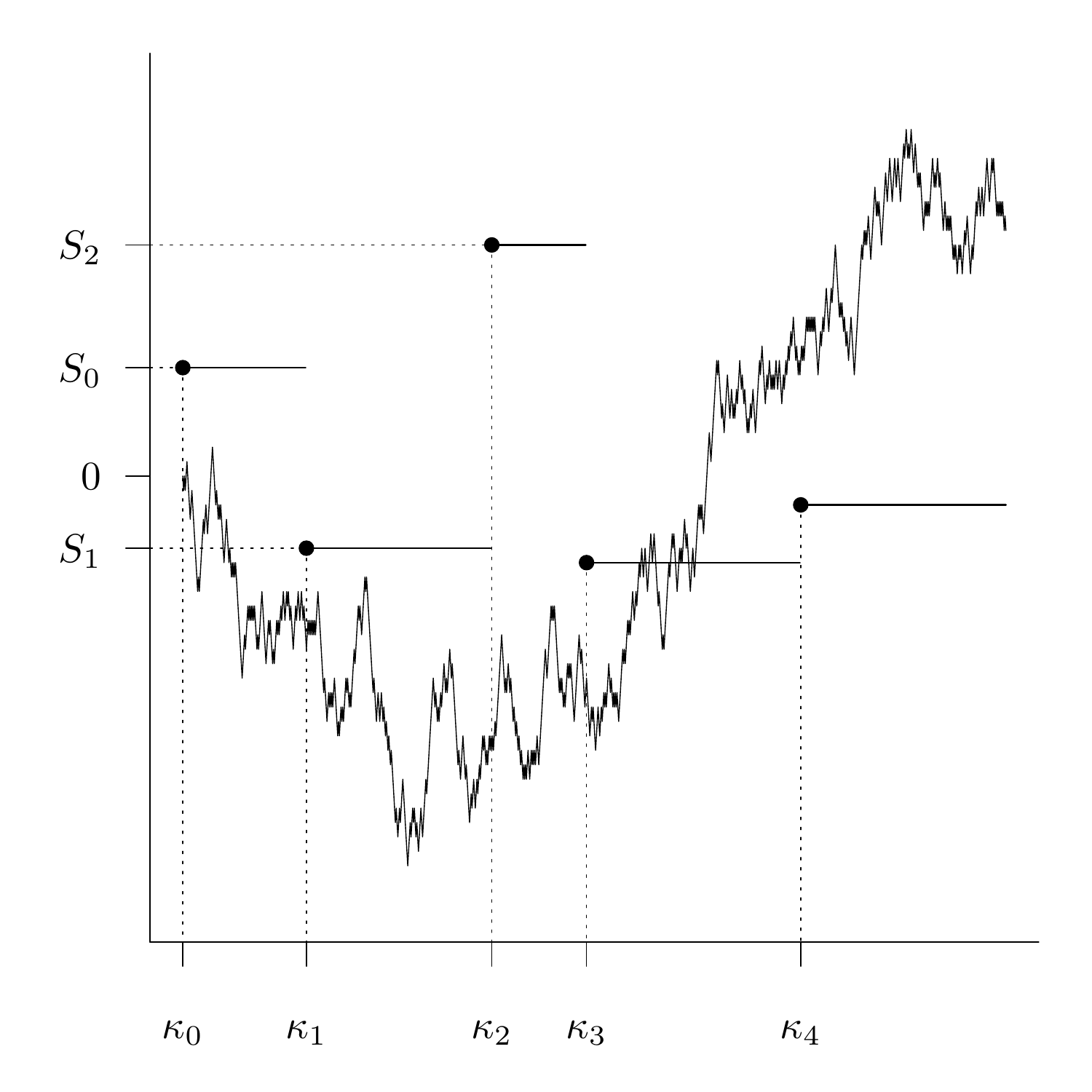}% 
    \includegraphics[width=0.45\textwidth]{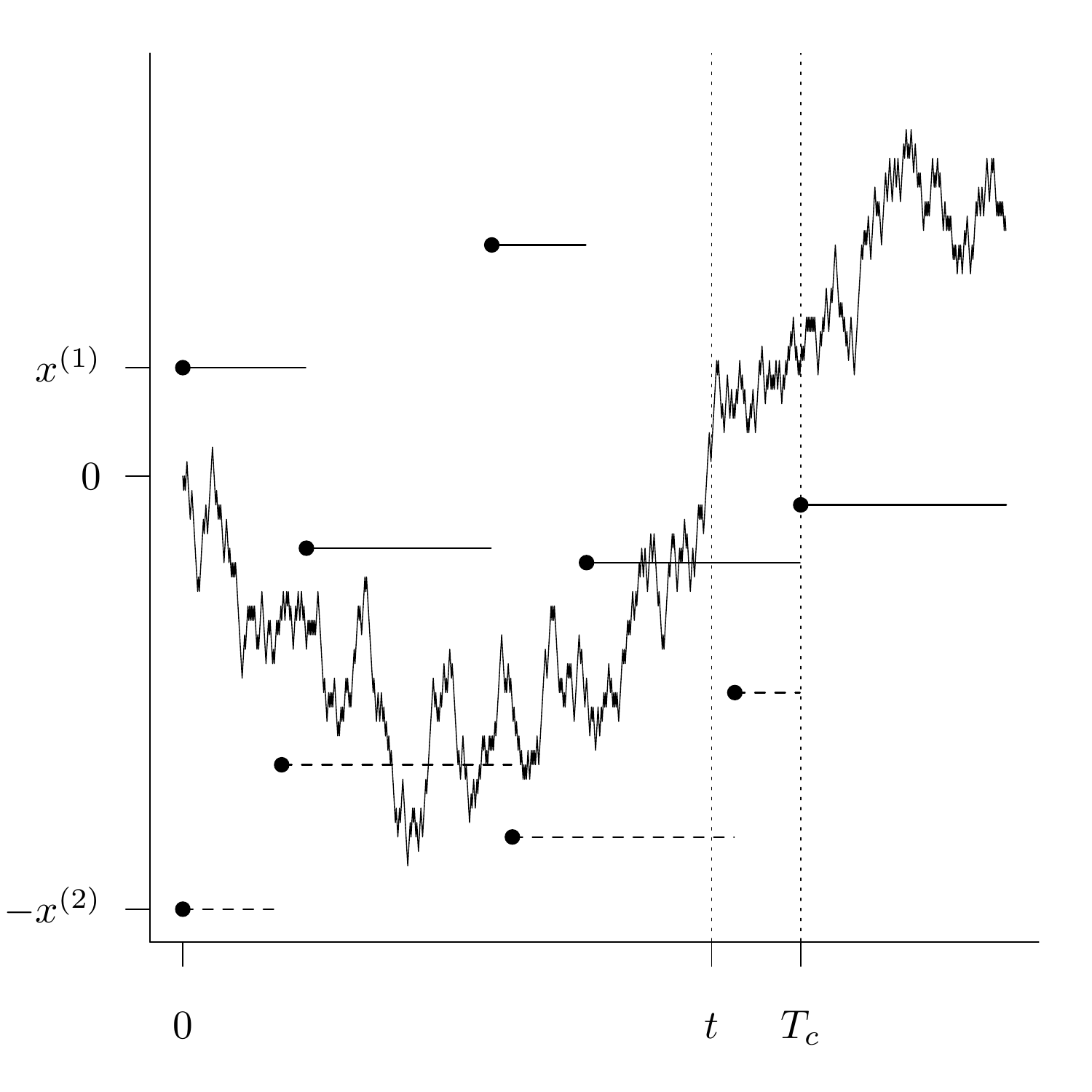}% 
    \caption{\label{fig:coupling} (A) Alternative graphical
      construction of Model~II from the proof of
      Proposition~\ref{P5}. (B) Graphical construction of two coupled
      $\gamma/\delta$-broken Brownian ratchet, Model~II, using the
      same Brownian motion, similar to (A). After time
      $\text{T}_{\text{coupl}}$ the active points and the jump times
      of active points are the same for both processes.}
  \end{center}
\end{figure}

\subsection{Proof of Theorem~\ref{T:speedII}}
\label{sub:proofT2}
The proof of Theorem~\ref{T:speedII} comes in three steps. First, we
will connect the speed of the broken Brownian ratchet, Model~II, to
the Markov chain $(Y_n, W_n,\eta_n)_{n\geq 1}$ from
Definition~\ref{Y_W_eta}, and show that the speed equals the ratio of
expected space- and time-increments between jump times. In Step~2, we
use the unique equilibrium of this Markov chain to obtain a fixed
point equation for $Y$, which allows to study its density. In Step~3,
we use insights into the density of $Y$ in order to study expected
increments between jump times.

~

\noindent {\bf Step 1: Connection to the Markov chain \boldmath$(Y_n, W_n,
  \eta_n)$}: In the last section we have established the unique
equilibrium distribution $\pi$ for the Markov chain $(Y_n, W_n,
\eta_n)_{n\geq 1}$. Let $(Y,W,\eta)$ be distributed according to $\pi$
and recall that $Y$ is the distance of the Brownian ratchet to its
reflection boundary at jump times, $W$ is the amount by which the
reflection boundary jumps and $\eta$ is the time between jumps. Note
that $\eta$ equals the killing time of the Brownian motion as
introduced in Definition~\ref{def:killedBM}, if the Brownian motion is
started in an independent copy of $Y$.

The proof works along similar lines as in the proof of Theorem~1
in~\cite{depper_pfaffel:2010}. Formally, Model~II is a cumulative
process similar to the thinned Model~I; see
Remark~\ref{remark_LLN}. Renewal points are times when the Brownian
ratchet hits the reflection boundary, given a jump of the reflection
boundary occurred between renewal points. The law of large numbers for
cumulative processes establishes that $X_t/t$ converges as
$t\to\infty$ to the ratio of expected space- and time increments
between renewal times, almost surely; compare with \eqref{eq:lln1}
below. Then, proceeding exactly as in the proof of Theorem~1
in~\citet[][p.~923]{depper_pfaffel:2010}, it can be seen from the ratio limit
theorem \citep[see e.g.][Thm.~6.6 on p. 231]{Revuz:1984} that 
% the ratio of
% expected increments between renewal times equals the ratio of
% expectation between jump points, hence
\begin{align}
  \label{eq:lln1}
  \lim_{t\to\infty} \frac{X_t}{t} = \frac{\mathbb E[W]}{\mathbb E[\eta]}. 
\end{align}

~

\noindent {\bf Step 2: A distributional identity and some
  consequences}: The triple $(Y,W,\eta)$ is distributed according to
the equilibrium $\pi$. Starting the Brownian ratchet in position $Y$,
the reflection boundary at the next jump time either jumps up or
down. Let $U$ be uniformly distributed on $[0,1]$ and denote the
length of the possible downward jump by $Z$, distributed as
$\mathrm{Exp}\left(\delta/\gamma\right)$. By construction
\begin{align}\label{Y_jump}
  Y\stackrel{d}{=}\ind_{\{\text{jump up at $\eta$}\}}
  \widetilde{B}^Y_{\eta-}U + \ind_{\{\text{jump down at $\eta$}\}}
  \left(\widetilde{B}^Y_{\eta-}+Z\right).
\end{align} 
Conditioned on the event that the reflection boundary jumps at time
$\eta$ the probability of an upward jump is
\begin{align*}
  \frac{\gamma\widetilde{B}^Y_{\eta-}}{\gamma\widetilde{B}^Y_{\eta-}+\delta}
  =
  \frac{\widetilde{B}^Y_{\eta-}}{\widetilde{B}^Y_{\eta-}+\delta/\gamma}
\end{align*}
and for a downward jump
\begin{align*}
  \frac{\delta}{\gamma\widetilde{B}^Y_{\eta-}+\delta} =
  \frac{\delta/\gamma}{\widetilde{B}^Y_{\eta-}+\delta/\gamma}.
\end{align*}
Let us write $f_X$ for the density of a random variable
$X$. Equations~(\ref{Y_jump}) and~\eqref{fy} imply
\begin{align*}
  f_Y(z) &=\int_0^\infty \!\!\!\!f_Y(x)\int_0^\infty \!\!\!\!
  f_{\widetilde{B}^x_{\eta-}}(u) \left(\frac{u}{u+\delta/\gamma}
    f_U\left(\frac{z}{u}\right)\frac{1}{u}
    +\frac{\delta/\gamma}{u+\delta/\gamma}f_Z(z-u)\right)\, du\,dx \\
  &=\int_0^\infty f_Y(x)\int_z^\infty f_{\widetilde B^x_{\eta-}}(u)
  \frac{1}{u+\delta/\gamma}\, du\,dx \\ & \qquad \qquad \qquad \qquad
  + \int_0^\infty f_Y(x)\int_0^z
  f_{\widetilde{B}^x_{\eta}}(u)\frac{\delta/\gamma}{u+\delta/\gamma}
  \frac\gamma{\delta} e^{-(z-u)\gamma/\delta} \,du \,dx  \\
  &= 2\gamma \underbrace{\int_0^\infty f_Y(x)\int_z^\infty G(x,u)\, du
    \, dx}_{=:A(z)} \\ & \qquad \qquad \qquad \qquad + 2\gamma
  \underbrace{\int_0^\infty f_Y(x)\int_0^z
    G(x,u)e^{-(z-u)\gamma/\delta} du
    \;dx}_{=:B(z)}.
\end{align*}
% where the last equality is just the definition of functions $A$ and
% $B$.  and for the first to last equality we used that
% $f_{\widetilde{B}^x_{\eta-}}(u)=G(x,u)(u+\tilde\delta)$ for the
% Green function $G$ from \eqref{Green_new}.
In particular we have
\begin{align}
  \label{eq:dens0}
  f_Y(0) = 2\gamma A(0) = 2\gamma \int_0^\infty f_Y(x)\int_0^\infty
  G(x,u)\,du\,dx.
\end{align}
We aim at showing that the pair $(A,B)$ satisfies the differential equation
\eqref{eq:diff-system}. For that we need to differentiate $A$ and $B$.
  The
derivatives of $A$ and $B$ are readily computed as 
\begin{align*}
  A'(z)& =
  -\frac{1}{w}\left[\psi(z)\int_z^\infty{f_Y(x)\phi(x)dx}+\phi(z)\int_0^z{f_Y(x)\psi(x)dx}\right]
  \\ \text{and} \quad B'(z) &=-A'(z) -\frac{\gamma}{\delta}B(z).
\end{align*}
Thus,  
\begin{align}
  \label{eq:1}
  f'_Y(z) & =2\gamma(A'(z)+B'(z))=-\frac{2\gamma^2}{\delta}B(z).
\end{align}
For $A$ we have  
\begin{align} \label{eq:6}
  A(z) & = \int_0^\infty f_Y(x)\int_z^\infty G(x,u)\,du\,dx,   \\
  \notag
  A'(z) &
  =-\frac{1}{w}\left[\psi(z)\int_z^\infty{f_Y(x)\phi(x)dx}+\phi(z)\int_0^z{f_Y(x)\psi(x)dx}\right],
  \\ \label{eq:8} A''(z) & =
  -\frac{1}{w}\left[\psi'(z)\int_z^\infty{f_Y(x)\phi(x)dx}+\phi'(z)\int_0^z{f_Y(x)\psi(x)dx}\right],
  \\ \notag
  \begin{split}
    A'''(z) &
    =-\frac{1}{w}\Bigl[\psi''(z)\int_z^\infty{f_Y(x)\phi(x)dx}+\phi''(z)\int_0^z{f_Y(x)\psi(x)dx}
    \\ & \qquad - f_Y(z) (\psi'(z)\phi(z) -\psi(z)\phi'(z)) \Bigr]\\
    & = f_Y(z) + (2\gamma z + 2\delta)A'(z),
  \end{split}
\end{align}
where we have used that $\phi$ and $\psi$ are solutions of
\eqref{2nd_operator2}, the equation for $A'$ and that the Wronskian
does not depend on $z$. Writing now $f_Y=2\gamma(A+B)$ and using
\eqref{eq:1} we obtain
\begin{align*}
  A'''(z) & = 2\gamma(A(z) +B(z)) +(2\gamma z + 2\delta) A'(z) \\ & =
  - 2\delta (A'(z)
  +B'(z)) + 2\gamma A(z) + (2\gamma z + 2\delta) A'(z) \\
  &= - 2\delta B'(z) + \left( 2\gamma z A(z)\right)'.
\end{align*}
Integrating both sides we have
\begin{align} \label{eq:12} A''(z) &= - 2\delta B(z) + 2\gamma z
  A(z).
\end{align} 
The integration constant vanishes because $B(0)=0$ and $A''(0)=0$, as
can be seen from \eqref{eq:8} (recall that $\psi'(0)=0$). Thus,
$(A,B)$ is a solution of \eqref{eq:diff-system}.

~

\noindent {\bf Step 3: Computation of \boldmath$\mathbb E[W]$ and $\mathbb
  E[\eta]$}: In the case that the reflection boundary jumps up $W$ and
$Y$ have the same distribution. If the reflection boundary jumps down
$W$ is distributed as $-Z$ where $Z$ is exponentially distributed with
parameter $\gamma/\delta$. Thus, the density of the invariant
distribution of $W$ is given by
\begin{align*}
  f_W(z) & = \int_0^\infty f_Y(x)\int_0^\infty
  f_{\widetilde{B}^x_{\eta-}}(u) \Big(\frac{u}{u+\delta/\gamma}
  \ind_{\{z \ge 0\}}f_U\left(\frac{z}{u}\right)\frac{1}{u} \\ & \qquad
  \qquad \qquad \qquad \qquad \qquad \qquad
  +\frac{\delta/\gamma}{u+\delta/\gamma}f_Z(-z) \ind_{\{z<0\}}\Big)\,
  du\,dx  \\
  & = 2\gamma A(z) \ind_{\{z\ge 0\}} + \int_0^\infty
  f_Y(x)\int_0^\infty f_{\widetilde{B}^x_{\eta-}}(u)
  \frac{\delta/\gamma}{u+\delta/\gamma} \frac\gamma{\delta}
  e^{z\gamma/\delta}  \ind_{\{z<0\}}\, du\,dx\\
  & = 2\gamma A(z) \ind_{\{z\ge 0\}} + e^{z\gamma/\delta}
  \ind_{\{z<0\}}
  \int_0^\infty f_Y(x)\int_0^\infty G(x,u) \, du\,dx \\
  & = 2\gamma A(z) \ind_{\{z\ge 0\}} + 2\gamma e^{z\gamma/\delta}
  \ind_{\{z<0\}} A(0).
\end{align*}
In order to compute $\mathbb E[W]$ we have to integrate the last
equations. We use \eqref{eq:12} and \eqref{eq:1} and obtain
\begin{equation}
  \label{eq:W}
  \begin{aligned}
    \mathbb E[W] & = \int_0^\infty 2\gamma z A(z) - 2\gamma z
    e^{-z\gamma/\delta} A(0)dz \\ & = \int_0^\infty A''(z) + 2\delta
    B(z) dz - \frac{2\delta^2}{\gamma} A(0) \\ & = -A'(0) +
    \frac{\delta^2}{\gamma^2}f_Y(0)- \frac{\delta^2}{\gamma^2} f_Y(0)
    = -A'(0).
  \end{aligned}
\end{equation}
Moreover, by \eqref{eq:dens0} we have
\begin{align}
  \label{eq:eta}
  \mathbb E[\eta] = 2 \int_0^\infty f_Y(x) \int_0^\infty G(x,y) \, dy
  \, dx = 2 A(0).
\end{align} 
Here the factor $2$ shows up because we have to integrate the Green
function with respect to the speed measure, which is $m(dy) = 2 \,dy$
in the case of Brownian motion.

Now, Theorem~\ref{T:speedII} follows from \eqref{eq:lln1},
\eqref{eq:W} and \eqref{eq:eta}, since multiplying $A$ by a constant
factor does not change the ratio $-A'(0)/(2A(0))$.

%\bibliographystyle{spbasic}
%\bibliography{DKP}

\end{document}